\documentclass[11pt]{article}

\usepackage{amssymb}
\usepackage{amsmath}
\usepackage{color}
\usepackage{rotating}
\usepackage{theorem}
\usepackage[all]{xy}
\usepackage{pgf}

\newcommand{\XYMATRIX}{\xymatrix@M=6pt}
\newcommand{\aremb}{\ar@{^{(}->}}
\newcommand{\arembfrom}{\ar@{<-^{)}}}

\numberwithin{equation}{section}

\theorembodyfont{\slshape}

  \newtheorem{THM}{Theorem}[section]
  \newtheorem{LEM}[THM]{Lemma}
  \newtheorem{PROP}[THM]{Proposition}
  \newtheorem{COR}[THM]{Corollary}

\theorembodyfont{\rmfamily}

  \newtheorem{DEF}[THM]{Definition}
  \newtheorem{EX}{Example}[section]

  \newtheorem{OBS}[THM]{Observation}

\newif\ifQEDsign
\newcommand{\QED}{\global\QEDsigntrue\hfill$\square$}

\newenvironment{proof}%
    {\par\noindent\textit{Proof.}\global\QEDsignfalse}%
    {\ifQEDsign\else\QED\fi\par\bigskip\par}

\newcommand{\quotient}[2]{\genfrac{[}{]}{0pt}{}{#1}{#2}}
\newcommand{\place}[3]{\underset{\llap{\raisebox{-1.5mm}{\scriptsize$#1$#2 place}}\uparrow}{#3}}

\renewcommand{\le}{\leqslant}
\renewcommand{\ge}{\geqslant}

\newcommand{\0}{\varnothing}

\renewcommand{\sec}{\cap}
\renewcommand{\phi}{\varphi}
\renewcommand{\epsilon}{\varepsilon}

\newcommand{\BB}{\mathbf{B}}
\newcommand{\CC}{\mathbf{C}}
\newcommand{\DD}{\mathbf{D}}

\newcommand{\KK}{\mathbf{K}}

\newcommand{\NN}{\mathbb{N}}

\newcommand{\QQ}{\mathbb{Q}}

\newcommand{\union}{\cup}

\newcommand{\Boxed}[1]{\mbox{$#1$}}
\newcommand{\embedsto}{\hookrightarrow}
\newcommand{\id}{\mathrm{id}}
\newcommand{\ID}{\mathrm{ID}}
\newcommand{\Ob}{\mathrm{Ob}}

\newcommand{\LO}{\mathrm{LO}}
\newcommand{\Sym}{\mathrm{Sym}}
\newcommand{\Con}{\mathrm{Con}}

\newcommand{\Surj}{\mathrm{Surj}}
\newcommand{\Bii}{\mathit{2}}
\newcommand{\op}{\mathrm{op}}

\newcommand{\FinSetSurj}{\mathbf{FSS}}
\newcommand{\FinSetInj}{\mathbf{FSI}}
\newcommand{\FinBaSurj}{\mathbf{FBAS}}

\newcommand{\BA}{\mathbf{FBA}}
\newcommand{\OBA}{\mathbf{OFBA}}
\newcommand{\VV}{\mathbf{V}_{\mathit{fin}}}
\newcommand{\OV}{\mathbf{OV}_{\mathit{fin}}}

\newcommand{\calA}{\mathcal{A}}
\newcommand{\calB}{\mathcal{B}}
\newcommand{\calC}{\mathcal{C}}
\newcommand{\calD}{\mathcal{D}}
\newcommand{\calF}{\mathcal{F}}

\newcommand{\calK}{\mathcal{K}}

\newcommand{\calM}{\mathcal{M}}
\newcommand{\calO}{\mathcal{O}}

\newcommand{\calQ}{\mathcal{Q}}

\newcommand{\calS}{\mathcal{S}}

\newcommand{\Fraisse}{Fra\"\i ss\'e}

\newcommand{\oa}{\overline{a}}

\newcommand{\LW}{L_{\omega_1,\omega}}
\newcommand{\leftexp}[2]{{\vphantom{#2}}^{#1}{#2}}
\newcommand{\W}{\leftexp{\omega >}{\omega}}
\DeclareMathOperator{\Flim}{Flim}
\DeclareMathOperator{\Aut}{Aut}

\DeclareMathOperator{\age}{age}

\title{Categorical Constructions and the Ramsey Property}
\author{Dragan Ma\v sulovi\'c\\
        University of Novi Sad, Faculty of Sciences\\
        Department of Mathematics and Informatics\\
        Trg Dositeja Obradovi\'ca 3, 21000 Novi Sad, Serbia\\
        email: dragan.masulovic@dmi.uns.ac.rs
        \and
        Lynn Scow\\
        Vassar College\\
        Department of Mathematics and Statistics\\
	124 Raymond Ave., Box 500\\
        Poughkeepsie, NY 12604, USA\\
        email: lyscow@vassar.edu
}

\begin{document}
\maketitle

\begin{abstract}
  It has become obvious in recent development that the structural Ramsey property is a categorical property:
  it depends not only on the choice of objects, but also on the choice of morphisms involved.
  In this paper we explicitly put the Ramsey property and the dual Ramsey property in the
  context of categories of finite structures and investigate the
  invariance of these properties under adjunctions and categorical equivalence.
  We use elementary category theory to generalize some combinatorial results and
  using the machinery of very basic category theory provide new combinatorial statements
  (whose formulations do not refer to category-theoretic notions).

  \bigskip

  \noindent \textbf{Key Words:} Ramsey property, categorical equivalence, topological dynamics

  \medskip

  \noindent \textbf{AMS Subj.\ Classification (2010):} 05C55, 18A99, 37B05
\end{abstract}

\section{Introduction}

It has become obvious in recent papers that
the structural Ramsey property is a categorical property:
it depends not only on the choice of objects, but also on the choice of morphisms involved
(see \cite{GLR,Nesetril,vanthe-ufcs,P-V,solecki-dual,zucker1,mu-pon}).
In this paper we explicitely put the Ramsey property and the
dual Ramsey property in the context of categories of finite structures and investigate the
invariance of these properties under adjunctions and categorical equivalence.
Our main result is that the Ramsey property is invariant under
categorical equivalence (and, hence, prove that the Ramsey property is a genuine categorical property).
We use elementary category theory to generalize some combinatorial results and
using the machinery of very basic category theory provide new combinatorial statements
(whose formulations do not refer to category-theoretic notions).
The intention of this paper can best be summarised by the following words od Bodirsky~\cite{Bodirsky}:
\begin{quote}
  ``Establishing that a class has the Ramsey property is often a substantial combinatorial challenge,
  and we are therefore interested in general transfer principles that allow to prove the Ramsey property
  by reducing to known Ramsey classes; this will be the topic of this text.''
\end{quote}

In Section~\ref{cerp.sec.prelim} we recall basics of category theory and formulate the
(dual) Ramsey property in terms of objects and morphisms (cf.~\cite{mu-pon,zucker1}).

In Section~\ref{cerp.sec.RPA} we first show that
right adjoints preserve the Ramsey property for morphisms while
left adjoints preserve the dual Ramsey property for morphisms.
The adjoints do not in general preserve the (dual) Ramsey property for objects,
so we move on to Ramsey properties in the context of categorical equivalence.
We show that if a category of finite structures has certain Ramsey property then the category
equivalent to it has the same Ramsey property, while the category dually
equivalent to it has the dual property. As an example we derive several Ramsey properties.

In Sections~\ref{cerp.sec.primal}, \ref{cerp.sec.fraisse} and~\ref{cerp.sec.oetd} we systematically treat
the best known and most important example of a categorical equivalence
in algebra, that between the variety of boolean algebras and any variety generated by
a single primal algebra (which is a finite algebra where all operations are term operations),~\cite{hu1,hu2}.

Section~\ref{cerp.sec.Lynn} contains a discussion of \Fraisse\ limits with identical automorphism group.
The principal motivation for this section is the following result from \cite{KPT} which states that
the Ramsey property is invariant under certain model-theoretic constructions and which is a special case
of our results in Section~\ref{cerp.sec.RPA}:

\begin{PROP} \cite[Proposition 9.1 (i)]{KPT}
  Let $\calK_0$ be a \Fraisse\ class in a signature $L_0$, let $L = L_0 \union \{\Boxed<\}$
  and let $\calK$, $\calK'$ be reasonable \Fraisse\ order classes in $L$ that are expansions of $\calK_0$.
  Assume that $\calK$ and $\calK'$ are simply bi-definable. The $\calK$ satisfies the Ramsey property
  iff $\calK'$ satisfies the Ramsey property.
\end{PROP}

We close the paper with two appendices.
The first one (Section~\ref{cerp.sec.prod}) generalizes
the product Ramsey theorem of M.~Soki\'c~\cite{sokic2,sokic-phd}.
We provide an abstract proof that if two categories have
some of Ramsey property, then their categorical product has the ``combined'' kind of Ramsey
property. As a consequence we have that for every category $\CC$ and every positive integer $n$,
if $\CC$ a (dual) Ramsey property then so does $\CC^n$, showing thus a
metaresult that every finite (dual) Ramsey theorem has the finite product version.
The second appendix (Section~\ref{cerp.sec.app2}) provides a discussion of \Fraisse\ theory
and the Kechris-Pestov-Todor\v cevi\'c correspondence in case of first-order structures
over a language which contains functional symbols.

\section{Preliminaries}
\label{cerp.sec.prelim}

\paragraph{Categories of structures.}
In order to specify a category $\CC$ one has to specify
a class of objects $\Ob(\CC)$,
a set of morphisms $\hom_\CC(\calA, \calB)$ for all $\calA, \calB \in \Ob(\CC)$,
an identity morphism $\id_\calA$ for all $\calA \in \Ob(\CC)$, and
the composition of morphisms $\cdot$ so that
\begin{itemize}
\item $(f \cdot g) \cdot h = f \cdot (g \cdot h)$, and
\item $\id_\calB \cdot f = f \cdot \id_\calA$ for all $f \in \hom_\CC(\calA, \calB)$.
\end{itemize}
Let $\Aut(\calA)$ denote the set of all invertible morphisms $\calA \to \calA$.
Recall that an object $\calA \in \Ob(\CC)$ is \emph{rigid} if $\Aut(\calA) = \{\id_\calA\}$.

For $\calA, \calB \in \Ob(\CC)$ we write $\calA \to \calB$ to denote that $\hom_\CC(\calA, \calB) \ne \0$.
Note that morphisms in $\hom_\CC(\calA, \calB)$ are not necessarily structure-preserving mappings
from $\calA$ to $\calB$, and that the composition
$\cdot$ in a category is not necessarily composition of mappings. We shall see examples later.
Instead of $\hom_\CC(\calA, \calB)$ we write $\hom(\calA, \calB)$ whenever $\CC$ is obvious from the context.

In this paper we are mostly interested in categories of structures.
%%Dragan: 9 Nov
  A \emph{structure} $\calA = (A, \Delta)$ is a set $A$ together with a set $\Delta$ of
  functions and relations on $A$, each having some finite arity.
  An embedding $f: \calA \rightarrow \calB$ is an injection $f: A \rightarrow B$ which respects
  the functions in $\Delta$, and respects and reflects the relations in $\Delta$.
%%Dragan: end
Surjective embeddings are isomorphisms. A structure $\calA$ is a \emph{substructure} of a structure
$\calB$ ($\calA \le \calB$) if the identity map is an embedding of $\calA$ into $\calB$.
%%Dragan: 9 Nov -- I commented out this sentence because in this paper we also consider surjections
  %In considering a category of (finite) structures we take ($\Delta$-)embeddings as morphisms.
%%DRagan: end
%some additional data $\Delta$ whose nature may vary from case to case (relational structures, topological spaces, algebraic structures etc).
%Structures usually come with clearly defined notions of substructure, isomorphism, embedding, surjective morphism etc.
Here is some further notation and terminology.
A structure $\calA = (A, \Delta)$ is finite if $A$ is a finite set.
The underlying set of a structure $\calA$, $\calA_1$, $\calA^*$, \ldots\ will always be denoted by its roman letter $A$, $A_1$, $A^*$, \ldots\ respectively.
We say that a structure $\calA = (A, \Delta)$ is \emph{ordered} if there is a binary relation
$<$ in $\Delta$ which linearly orders $A$.
Given a structure $\calA = (A, \Delta)$ and a linear ordering $<$ on $A$,
we write $\calA_<$ for the structure $(A, \Delta, \Boxed<)$.
Moreover, we shall always write $\calA$ to denote the unordered reduct of~$\calA_<$.
Linear orders denoted by $<$, $\sqsubset$ etc.\ are irreflexive (strict linear orders), whereas
by $\le$, $\sqsubseteq$ etc.\ we denote the corresponding reflexive linear orders.

\paragraph{Adjunction, equivalence and isomorphism of categories.}
A pair of functors $F : \CC \rightleftarrows \DD : G$ is an \emph{adjunction} provided there is a family of isomorphisms
$\Phi_{\calC,\calD} : \hom_\DD(F(\calC), \calD) \cong \hom_\CC(\calC, G(\calD))$ natural in both $\calC$ and $\calD$.
We say that $F$ is \emph{left adjoint} to $G$ and $G$ is \emph{right adjoint} to $F$. Every adjunction
$F : \CC \rightleftarrows \DD : G$ gives rise to two natural transformations
$\eta : \ID_\CC \to GF$ and $\epsilon : FG \to \ID_\DD$ referred to as \emph{unit} and \emph{counit}, respectively,
satisfying the so-called \emph{unit-counit identities}
$\epsilon F \cdot F \eta = \id_F$ and $G \epsilon \cdot \eta G = \id_G$. If $f : F(\calC) \to \calD$ and
$g : \calC \to G(\calD)$ are morphisms in $\DD$ and $\CC$, respectively, then
$\Phi(f) = G(f) \cdot \eta_\calC$ and $\Phi^{-1}(g) = \epsilon_\calD \cdot F(g)$.

Categories $\CC$ and $\DD$ are \emph{equivalent} if there exist functors $E : \CC \to \DD$
and $H : \DD \to \CC$, and natural isomorphisms $\eta : \ID_\CC \to HE$ and $\epsilon : \ID_\DD \to EH$.
We say that $H$ is a pseudoinverse of $E$ and vice versa.
It is a well known fact that a functor $E : \CC \to \DD$ has a pseudoinverse if and only if it is full, faithful and isomorphism-dense.
If $E$ has a pseudoinverse then $\CC$ and $\DD$ are equivalent. Clearly, categorical equivalence is a particular form of adjunction.

A \emph{skeleton} of a category is a full, isomorphism-dense subcategory in which no two distinct objects are isomorphic.
It is easy to see that (assuming (AC)) every category has a skeleton. It is also a well known fact that two categories are
equivalent if and only if they have isomorphic skeletons.
Categories $\CC$ and $\DD$ are \emph{dually equivalent} if $\CC$ and $\DD^\op$ are equivalent.

Categories $\CC$ and $\DD$ are \emph{isomorphic} if there exist functors $E : \CC \to \DD$ and $H : \DD \to \CC$ such that
$H$ is the inverse of $E$. A functor $E : \CC \to \DD$ is isomorphism-dense if
for every $D \in \Ob(\DD)$ there is a $C \in \Ob(\CC)$ such that $E(C) \cong D$.

\paragraph{Ramsey property for categories.}
We say that
$
  \calS = \calM_1 \union \ldots \union \calM_k
$
is a \emph{$k$-coloring} of $\calS$ if $\calM_i \sec \calM_j = \0$ whenever $i \ne j$.
Equivalentnly, a $k$-coloring of $\calS$ is a mapping $\chi : \calS \to \{1, 2, \ldots, k\}$.
We shall use both points of view as we find appropriate.

Given a category $\CC$
let $\hom(\calA, \calB)$ denote the set of all morphisms $\calA \to \calB$ in $\CC$. Define $\sim_\calA$ on
$\hom(\calA, \calB)$ as follows: for $f, f' \in \hom(\calA, \calB)$ we let $f \sim_\calA f'$ if $f' = f \cdot \alpha$
for some $\alpha \in \Aut(\calA)$. Then
$$
  \binom \calB \calA = \hom(\calA, \calB) / \Boxed{\sim_\calA}
$$
corresponds to all subobjects of $\calB$ isomorphic to $\calA$ (see \cite{Nesetril,NR-2}).
For an integer $k \ge 2$ and $\calA, \calB, \calC \in \Ob(\CC)$ we write
$$
  \calC \longrightarrow (\calB)^{\calA}_k
$$
to denote that $\calA \to \calB \to \calC$ and for every $k$-coloring
$
  \binom \calC\calA = \calM_1 \union \ldots \union \calM_k
$
there is an $i \in \{1, \ldots, k\}$ and a morphism $w : \calB \to \calC$ such that
$w \cdot \binom \calB\calA \subseteq \calM_i$.
(Note that $w \cdot (f / \Boxed{\sim_\calA}) = (w \cdot f) / \Boxed{\sim_\calA}$ for $f / \Boxed{\sim_\calA} \in \binom\calB\calA$.)
We write
$$
  \calC \overset{\mathit{hom}}\longrightarrow (\calB)^{\calA}_k
$$
to denote that $\calA \to \calB \to \calC$ in $\CC$ and for every $k$-coloring
$
  \hom(\calA, \calC) = \calM_1 \union \ldots \union \calM_k
$
there is an $i \in \{1, \ldots, k\}$ and a morphism $w : \calB \to \calC$ such that
$w \cdot \hom(\calA, \calB) \subseteq \calM_i$.

A category $\CC$ has the \emph{Ramsey property for objects} if
for every integer $k \ge 2$ and all $\calA, \calB \in \Ob(\CC)$
such that $\calA \to \calB$ there is a
$\calC \in \Ob(\CC)$ such that $\calC \longrightarrow (\calB)^{\calA}_k$.
A category $\CC$ has the \emph{Ramsey property for morphisms} if
for every integer $k \ge 2$ and all $\calA, \calB \in \Ob(\CC)$
such that $\calA \to \calB$ there is a
$\calC \in \Ob(\CC)$ such that $\calC \overset{\mathit{hom}}\longrightarrow (\calB)^{\calA}_k$.

In a category of finite ordered structures all the relations
$\sim_\calA$ are trivial and the two Ramsey properties coincide.
Therefore, we say that a category of finite ordered structures and embeddings has the \emph{Ramsey property}
if it has the Ramsey property for morphisms.

\begin{EX}\label{cerp.ex.FRP}
    The category $\FinSetInj$ of finite sets and injective maps has the Ramsey property for objects.
    This is just a reformulation of the Finite Ramsey Theorem:

  \begin{THM}\label{cerp.thm.FRT} \cite{Ramsey}
    For all positive integers $k$, $a$, $m$ there is a positive integer $n$ such that
    for every $n$-element set $C$ and every $k$-coloring of the set $\binom Ca$ of all $a$-element subsets of $C$
    there is an $m$-element subset $B$ of $C$ such that $\binom Ba$ is monochromatic.
  \end{THM}
\end{EX}

A category $\CC$ has the \emph{dual Ramsey property for objects (morphisms)} if
$C^\op$ has the Ramsey property for objects (morphisms).

\begin{EX}\label{cerp.ex.FDRT}
    The category $\FinSetSurj$ of finite sets and surjective maps has the dual Ramsey property for objects.
    This is just a reformulation of the Finite Dual Ramsey Theorem:

  \begin{THM}\label{cerp.thm.FDRT} \cite{GR}
    For all positive integers $k$, $a$, $m$ there is a positive integer $n$ such that
    for every $n$-element set $C$ and every $k$-coloring of the set $\quotient Ca$ of all partitions of
    $C$ with exactly $a$ blocks there is a partition $\beta$ of $C$ with exactly $m$ blocks such that
    the set of all patitions from $\quotient Ca$ which are coarser than $\beta$ is monochromatic.
  \end{THM}

    To show that this is indeed the case, let $\CC = \FinSetSurj$. For $A, B \in \Ob(\CC)$
    let $\Surj(B, A)$ denote the set of all surjective maps $B \twoheadrightarrow A$.
    Define $\equiv_A$ on $\Surj(B, A)$ as follows: for $f, f' \in \Surj(B, A)$ we let
    $f \equiv_A f'$ if $f' = \alpha \circ f$ for some bijection~$\alpha : A \to A$.

    The claim that $\CC$ has the dual Ramsey property for objects means that
    $\CC^\op$ has the Ramsey property for objects, which then means that
    for every integer $k \ge 2$ and all $A, B \in \Ob(\CC)$
    such that $\hom_{\CC^\op}(A, B) \ne \0$ there is a $C \in \Ob(\CC)$ such that for every $k$-coloring
    $
      {\binom CA}_{\CC^\op} = \calM_1 \union \ldots \union \calM_k
    $
    there is an $i \in \{1, \ldots, k\}$ and a morphism $w \in \hom_{\CC^\op}(B, C)$ such that
    $w \cdot {\binom BA}_{\CC^\op} \subseteq \calM_i$. Since
    \begin{align*}
      f \cdot g \text{ (in $\CC^\op$)} &= g \circ f\\
      \hom_{\CC^\op}(A, B) &= \hom_\CC(B, A) = \Surj(B, A)\\
      {\binom BA}_{\CC^\op} &= \hom_{\CC^\op}(A, B) \text{ factored by } \sim_A \text{ in } \CC^\op\\
                                    &= \hom_\CC(B, A) \text{ factored by } \equiv_A \text{ in } \CC\\
                                    &= \Surj(B, A) / \Boxed{\equiv_A},
    \end{align*}
    the fact that $\CC^\op$ has the Ramsey property for objects reads as follows:
    for every integer $k \ge 2$ and all finite sets $A$ and $B$
    such that $\Surj(B, A) \ne \0$ there is a finite set
    $C$ such that for every $k$-coloring
    $$
      \Surj(C, A) / \Boxed{\equiv_A} \; = \; \calM_1 \union \ldots \union \calM_k
    $$
    there is an $i \in \{1, \ldots, k\}$ and a surjective mapping
    $w \in \Surj(C, B)$ satisfying
    $(\Surj(B, A) / \Boxed{\equiv_A}) \circ w \subseteq \calM_i$.
    (Note that $(f / \Boxed{\equiv_A}) \circ w = (f \circ w) / \Boxed{\equiv_A}$ for $f / \Boxed{\equiv_A} \in
    \Surj(B, A) / \equiv_A$.)
    Since $\Surj(B, A) / \Boxed{\equiv_A}$ corresponds to partitions of $B$ into $|A|$-many blocks
    we see that the categorical statement is indeed a reformulation of Theorem~\ref{cerp.thm.FDRT}.
\end{EX}

We can show that the Ramsey property for objects and the Ramsey property for morphisms
are closely related for categories where all the morphisms are monic (that is, left cancellable;
compare with \cite{zucker1}).
%Lynn: Nov 18 -- I thought it was important to add some nod to Nesetril here, what do you think?
The assumption of rigidity below was pointed out in \cite{Nesetril}.
%Lynn: end

\begin{PROP}\label{cerp.prop.rigid}
  Let $\CC$ be a category where morphisms are monic.
  If $\CC$
  has the Ramsey property for morphisms then all the objects in $\CC$ are rigid. Consequently,
  a category $\CC$ has the Ramsey property for morphisms if and only if all the objects in $\CC$ are rigid
  and $\CC$ has the Ramsey property for objects.
\end{PROP}
\begin{proof}
  Assume that $\calA \in \Ob(\CC)$ is not rigid and let
  $\alpha \in \Aut(\calA)$ be an automorphism of $\calA$ such that $\alpha \ne \id_\calA$.
  In order to show that $\CC$ does not have the Ramsey property for morphisms, take any $\calC \in \Ob(\CC)$
  and let us show that $\calC \not\longrightarrow (\calA)^\calA_2$.

  Let $\langle \alpha \rangle$ be the cyclic group generated by~$\alpha$.
  Then $|\langle \alpha \rangle| \ge 2$ because $\alpha \ne \id_\calA$.
  Let $\langle \alpha \rangle$ act on $\hom(\calA, \calC)$
  by $h^\alpha = h \cdot \alpha$. The orbits of this action are of the form
  $h \cdot \langle \alpha \rangle$, where $h \in \hom(\calA, \calC)$.
  It follows that $|h \cdot \langle \alpha \rangle| = |\langle \alpha \rangle| \ge 2$ because $h$ is monic.

  Let $\chi : \hom(\calA, \calC) \to 2$ be any coloring of $\hom(\calA, \calC)$ such that
  $\chi$ assumes both colors on each orbit of the action of $\langle \alpha \rangle$ on $\hom(\calA, \calC)$.
  Then for every $w : \calA \to \calC$ we have that
  $|\chi(w \cdot \hom_\CC(\calA, \calA))| \ge |\chi(w \cdot \langle \alpha \rangle)| = 2$
  because $w \cdot \langle \alpha \rangle \subseteq w \cdot \hom(\calA, \calA)$ and $\chi$ assumes both colors on each orbit.
\end{proof}

The following are easy lemmas (cf.~\cite{zucker1}):

\begin{LEM}\label{cerp.lem.easy}
  $(a)$
  If $\calC \overset{\mathit{hom}}\longrightarrow (\calB)^{\calA}_k$ and $\calB_1 \to \calB$ then $\calC \overset{\mathit{hom}}\longrightarrow (\calB_1)^{\calA}_k$.

  $(b)$
  If $\calC \longrightarrow (\calB)^{\calA}_k$ and $\calB_1 \to \calB$ then $\calC \longrightarrow (\calB_1)^{\calA}_k$.
\end{LEM}

\begin{LEM}\label{cofinal}
  %%Dragan: 9 Nov
  Let $\CC$ be a category whose morphisms are monic.
  If $\CC$ is has the Ramsey property for morphisms (objects)
  and $\DD$ is a full subcategory of $\CC$ such that $\Ob(\DD)$ is cofinal in $\Ob(\CC)$,
  then $\DD$ has the Ramsey property for morphisms (objects).
  %%Dragan: end
\end{LEM}

\section{Ramsey property, adjunctions and categorical equivalence}
\label{cerp.sec.RPA}

In this section we discuss Ramsey properties in adjunctions and prove that, under certain reasonable assumptions,
right adjoints preserve the Ramsey property for morphisms, while left adjoints preserve the dual of the Ramsey property for morphisms.
The status of the Ramsey properties for objects is delicate and is preserved by right, respectively, left
adoints under additinal assumptions on the automorphism groups of objects of the form $F(\calC)$ and $G(\calD)$.
We then treat the status of Ramsey properties in case of categorical equivalence and prove that
Ramsey properties carry over from a category to its equivalent category.
Our main results in this section are a bit technical, but we can then show how to use them to infer several
Ramsey theorems and dual Ramsey theorems.

\begin{THM}\label{cerp.thm.adj}
  Let $F : \CC \rightleftarrows \DD : G$ be an adjunction.

  $(a)$ If $\DD$ has the Ramsey property for morphisms then so does $\CC$.

  $(b)$ If $\CC$ has the dual Ramsey property for morphisms then so does $\DD$.
\end{THM}
\begin{proof}
  It suffices to prove $(a)$ as the proof of $(b)$ is dual.
  Let $\Phi$ be the natural isomorphism between the hom-sets.

  Take any $k \ge 2$ and any $\calA, \calB \in Ob(\CC)$ such that $\calA \to \calB$. Then
  $F(\calA) \to F(\calB)$. Since $\DD$ has the Ramsey property for morphisms, there is
  a $\calC \in \Ob(\DD)$ such that $\calC \overset{\mathit{hom}}\longrightarrow (F(\calB))^{F(\calA)}_k$.
  Let us show that $G(\calC) \overset{\mathit{hom}}\longrightarrow (GF(\calB))^{\calA}_k$.
  Take any $k$-coloring
  $$
    \hom(\calA, G(\calC)) = \calM_1 \union \ldots \union \calM_k.
  $$
  By applying $\Phi^{-1}$ and having in mind that
  $\Phi^{-1}(\hom(\calA, G(\calC))) = \hom(F(\calA), \calC)$ we obtain
  $$
    \hom(F(\calA), \calC) = \Phi^{-1}(\calM_1) \union \ldots \union \Phi^{-1}(\calM_k).
  $$
  Since $\Phi$ is bijective, the above is actually a $k$-coloring of $\hom(F(\calA), \calC)$, so there is an $i$
  and a morphism $w : F(\calB) \to \calC$ such that
  $$
    w \cdot \hom(F(\calA), F(\calB)) \subseteq \Phi^{-1}(\calM_i).
  $$
  After applying $G$ and multiplying by $\eta_\calA$ from the right we have
  $$
    G(w) \cdot G(\hom(F(\calA), F(\calB))) \cdot \eta_\calA \subseteq G(\Phi^{-1}(\calM_i)) \cdot \eta_\calA.
  $$
  Loosely speaking, for any set of morphisms $\calM$ we have that $G(\calM) \cdot \eta = \Phi(\calM)$, so
  the above relation transforms to
  $$
    G(w) \cdot \Phi(\hom(F(\calA), F(\calB))) \subseteq \Phi(\Phi^{-1}(\calM_i)),
  $$
  or, equivalently,
  $$
    G(w) \cdot \hom(\calA, GF(\calB)) \subseteq \calM_i.
  $$
  This completes the proof that $G(\calC) \overset{\mathit{hom}}\longrightarrow (GF(\calB))^{\calA}_k$.

  Since $\eta_\calB : \calB \to GF(\calB)$, Lemma~\ref{cerp.lem.easy}~$(a)$ ensures that
  $G(\calC) \overset{\mathit{hom}}\longrightarrow (\calB)^{\calA}_k$. Therefore, $\CC$ has the Ramsey property for morphisms.
\end{proof}

The analogous statement for objects need not be true in general
becuase, in general, the unit and the counit do not consist of isomorphisms.
The following lemma provides a sufficient condition for adjunctions to preserve Ramsey property for objects
but the additional condition we impose is rather strong.

\begin{LEM}\label{cerp.lem.adj2}
  Let $F : \CC \rightleftarrows \DD : G$ be an adjunction.

  $(a)$ Assume that $\Aut(F(\calA)) = F(\Aut(\calA))$ for all
  $\calA \in \Ob(\CC)$. If $\DD$ has the Ramsey property for objects then so does $\CC$.

  $(b)$ Assume that $\Aut(G(\calB)) = G(\Aut(\calB))$ for all
  $\calB \in \Ob(\DD)$. If $\CC$ has the dual Ramsey property for objects then so does $\DD$.
\end{LEM}
\begin{proof}
  Again, we shall focus on $(a)$ because the proof of $(b)$ is dual. The proof of $(a)$, however, is analogous
  to the proof of $(a)$ in Theorem~\ref{cerp.thm.adj} provided we can show that
  $$
    \Phi\left(\binom{\calB}{F(\calA)}\right) = \binom{G(\calB)}{\calA},
  $$
  or, equivalently,
  $$
    \Phi\big(\hom(F(\calA), \calB) / \Boxed{\sim_{F(\calA)}}\big) = \hom(\calA, G(B)) / \Boxed{\sim_\calA}.
  $$
  This relation clearly follows from
  \begin{equation}\label{cerp.eq.adj}
    \Phi(f / \Boxed{\sim_{F(\calA)}}) = \Phi(f) / \Boxed{\sim_\calA} \text{ for all }
    f \in \hom(F(\calA), \calB).
  \end{equation}
  Let us show~\eqref{cerp.eq.adj}.

  $(\subseteq)$
  Take any $f \in \hom(F(\calA), \calB)$ and
  any $g \in \Phi(f / \Boxed{\sim_{F(\calA)}})$. Then $g = \Phi(f \cdot \alpha)$ for some $\alpha \in \Aut(F(\calA))$.
  Then $g = G(f) \cdot G(\alpha) \cdot \eta_\calA$. Since $\alpha \in \Aut(F(\calA)) = F(\Aut(\calA))$ there is
  a $\beta \in \Aut(\calA)$ such that $\alpha = F(\beta)$, so
  \begin{align*}
    g &= G(f) \cdot G(\alpha) \cdot \eta_\calA\\
      &= G(f) \cdot GF(\beta) \cdot \eta_\calA\\
      &= G(f) \cdot \eta_\calA \cdot \beta \text{\quad[because $\eta$ is natural]}\\
      &= \Phi(f) \cdot \beta \in \Phi(f) / \Boxed{\sim_\calA}.
  \end{align*}

  $(\supseteq)$
  Analogous to $(\subseteq)$.
\end{proof}

We thank Christian Rosendal for letting us include the following example from his unpublished notes.

\begin{figure}
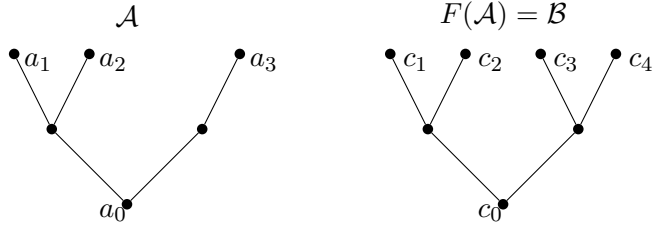

\begin{pgfpicture}{-4.5cm}{0cm}{0cm}{0cm}

\pgfputat{\pgfxy(-1,2.5)}{\pgfbox[center,center]{$\calA$}}

\pgfnodecircle{Node0}[fill]{\pgfxy(-1,0)}{.07cm}
\pgfnodecircle{Node1}[fill]{\pgfxy(-2,1)}{.07cm}
\pgfnodecircle{Node2}[fill]{\pgfxy(-2.5,2)}{.07cm}
\pgfnodecircle{Node3}[fill]{\pgfxy(-1.5,2)}{.07cm}

\pgfnodecircle{Node4}[fill]{\pgfxy(0,1)}{.07cm}
%\pgfnodecircle{Node5}[fill]{\pgfxy(0.5,2)}{.07cm}
\pgfnodecircle{Node6}[fill]{\pgfxy(0.5,2)}{.07cm}

\pgfnodeconnline{Node0}{Node1}
\pgfnodeconnline{Node0}{Node4}
\pgfnodeconnline{Node1}{Node2}
\pgfnodeconnline{Node1}{Node3}
%\pgfnodeconnline{Node4}{Node5}
\pgfnodeconnline{Node4}{Node6}

\pgfputat{\pgfxy(-1,0)}{\pgfbox[right,top]{$a_0$}}
\pgfputat{\pgfxy(-1,2)}{\pgfbox[right,top]{$a_2$}}
\pgfputat{\pgfxy(-2,2)}{\pgfbox[right,top]{$a_1$}}
%\pgfputat{\pgfxy(1,2)}{\pgfbox[right,top]{$c_3$}}
\pgfputat{\pgfxy(1,2)}{\pgfbox[right,top]{$a_3$}}

\pgfputat{\pgfxy(4,2.5)}{\pgfbox[center,center]{$F(\calA) = \calB$}}

\pgfnodecircle{Node20}[fill]{\pgfxy(4,0)}{.07cm}
\pgfnodecircle{Node21}[fill]{\pgfxy(3,1)}{.07cm}
\pgfnodecircle{Node22}[fill]{\pgfxy(2.5,2)}{.07cm}
\pgfnodecircle{Node23}[fill]{\pgfxy(3.5,2)}{.07cm}

\pgfnodecircle{Node24}[fill]{\pgfxy(5,1)}{.07cm}
\pgfnodecircle{Node25}[fill]{\pgfxy(4.5,2)}{.07cm}
\pgfnodecircle{Node26}[fill]{\pgfxy(5.5,2)}{.07cm}

\pgfnodeconnline{Node20}{Node21}
\pgfnodeconnline{Node20}{Node24}
\pgfnodeconnline{Node21}{Node22}
\pgfnodeconnline{Node21}{Node23}
\pgfnodeconnline{Node24}{Node25}
\pgfnodeconnline{Node24}{Node26}

\pgfputat{\pgfxy(4,0)}{\pgfbox[right,top]{$c_0$}}
\pgfputat{\pgfxy(4,2)}{\pgfbox[right,top]{$c_2$}}
\pgfputat{\pgfxy(3,2)}{\pgfbox[right,top]{$c_1$}}
\pgfputat{\pgfxy(5,2)}{\pgfbox[right,top]{$c_3$}}
\pgfputat{\pgfxy(6,2)}{\pgfbox[right,top]{$c_4$}}
\end{pgfpicture}

\caption{Unordered trees\label{untree}}
\end{figure}

%%Lynn: 18 Nov (just the first line)
%%Dragan: 9 Nov -- small edits
\begin{EX} (Homogeneous trees.)  Here we present an example that shows the importance of the assumption in Lemma \ref{cerp.lem.adj2}~(a).
%%Lynn: end
Let $\CC$ be the category of trees, finite structures in the language $\{f\}$ where $f(a) = b$ if $b$ is the immediate predecessor of $a$ in the partial tree order.  Let $\DD \subset \CC$ be the class of homogeneous trees, trees where the branching number at a node is a
function of the level of the node, i.e. for every $n < \omega$, there is some $b(n) < \omega$ such that every node at level $n$
has exactly $b(n)$ immediate successors at level $n+1$.  Let $\CC_<$, $\DD_<$ be the corresponding categories of finite
ordered trees (structures in the language $\{f, \Boxed<\}$ where $<$ gives the linear extension of the partial tree order).
It is known that $\CC_<$ has the Ramsey property for objects/morphisms, see \cite{fou} and see \cite{scow} for a discussion.  In Rosendal's notes it is shown that $\DD$ has the Ramsey property for objects.  This follows by Lemma \ref{cofinal} and \cite[Proposition 5.6]{KPT} as $\DD_<$ is a cofinal full subcategory of $\CC_<$ and it is order forgetful.

This example illustrates a case where we have an adjunction  $F : \CC \rightleftarrows \DD : G$ but     $\Phi\big(\hom(F(\calA), \calB) / \Boxed{\sim_{F(\calA)}}\big) = \hom(\calA, G(B)) / \Boxed{\sim_\calA}$
 fails.  $F$ takes $\calA \in \Ob(\CC)$ to the smallest homogeneous tree containing $\calA$ and $G$ gives the inclusion of $\DD$ in $\CC$.  Consider $\calA$ as given in Figure \ref{untree} and let $\calB = F(\calA)$.
Then $|\hom(F(\calA), \calB) / \Boxed{\sim_{F(\calA)}} |=1$ but $|\hom(\calA, G(B)) / \Boxed{\sim_\calA}| = 4$: a map can send $(a_1,a_2,a_3)$ to any of $(c_1,c_2,c_3), (c_1,c_2,c_4), (c_1,c_3,c_4), (c_2,c_3,c_4)$, up to automorphism of $\calA$.

Notice that this is also a case where $\DD$ has the Ramsey property for objects but $\CC$ does not.  Rosendal provides the example from Figure \ref{untree} in his notes.  If we impose an ordering $<$ on all members of $\CC$ and color copies of $\calA$ in $\calB$ according to whether the pair $(a_1,a_2)$ occurs $<$ then $a_3$ or vice versa, then we cannot find a homogeneous copy of $\calB$.
\end{EX}
%%Dragan: end

We shall now move on to categorical equivalence which behaves more nicely with respect to Ramsey properties.

\begin{THM}\label{cerp.thm.dual}
  Let $\CC$ and $\DD$ be equivalent categories.
  Then $\CC$ has the Ramsey property for objects (morphisms) if and only if $\DD$ does.

  In particular, if $\CC$ and $\DD$ are dually equivalent and one of them has the Ramsey property for
  morphisms (objects), the other has the dual Ramsey property for morphisms (objects).
\end{THM}
\begin{proof}
  Let us prove the statement in case of objects as the proof in case of
  morphisms is analogous and follows from Theorem~\ref{cerp.thm.adj}.

  Let $E : \CC \to \DD$ and $H : \DD \to \CC$ be functors that constitute the equivalence between $\CC$ and $\DD$
  and let $\eta : \ID_\CC \to HE$ and $\epsilon : \ID_\DD \to EH$ be the accompanying natural isomorphisms.
  Assume that $\DD$ has the Ramsey property and let us show that $\CC$ has the Ramsey property (the other direction is analogous).
  Take any positive integer $k$
  and let $\calA \to \calB$ in $\CC$. Then $E(\calA) \to E(\calB)$ in $\DD$,
  so there is a $\calC \in \Ob(\DD)$ such that
  \begin{equation}\label{gaif.eq.RP1}
    \calC \longrightarrow (E(\calB))^{E(\calA)}_k.
  \end{equation}
  Let us show that
  $
    H(\calC) \longrightarrow (\calB)^{\calA}_k
  $
  in $\CC$. Note first that $\calB \to H(\calC)$ because $E(\calB) \to \calC$, whence
  $\calB \cong HE(\calB) \to H(\calC)$. Let
  $$
    \binom{H(\calC)}{\calA} = \calM_1 \union \ldots \union \calM_k
  $$
  be an arbitrary $k$-coloring. Let
  $$
    \calM_i^E = \{ E(f) / \Boxed{\sim_{E(\calA)}} : f / \Boxed{\sim_{\calA}} \in \calM_i\}.
  $$
  Then it is easy to show that
  $$
    \binom{EH(\calC)}{E(\calA)} = \calM_1^E \union \ldots \union \calM_k^E
  $$
  is a $k$-coloring. Having in mind that
  $\epsilon_{\calC} : \calC \to EH(\calC)$ is an isomorphism, we have that
  $$
    \binom{\calC}{E(\calA)}
      = \epsilon_{\calC}^{-1} \cdot \calM_1^E  \union \ldots \union \epsilon_{\calC}^{-1} \cdot \calM_k^E
  $$
  is also a $k$-coloring. From \eqref{gaif.eq.RP1} we know that there is a
  $w : E(\calB) \to \calC$ and a color $i$ such that
  \begin{equation}\label{gaif.eq.RP2}
    w \cdot \binom{E(\calB)}{E(\calA)} \subseteq \epsilon_{\calC}^{-1} \cdot \calM_i^E.
  \end{equation}
  Let $w^* = H(w) \cdot \eta_{\calB} : \calB \to H(\calC)$ and let us show that
  \begin{equation}\label{gaif.eq.RP3}
    w^* \cdot \binom{\calB}{\calA} \subseteq \calM_i.
  \end{equation}
  Take any $u / \Boxed{\sim_\calA} \in \binom{\calB}{\calA}$. Then
  \begin{align*}
    w^* \cdot (u / \Boxed{\sim_\calA})
    & = (w^* \cdot u) / \Boxed{\sim_\calA}\\
    &= (H(w) \cdot \eta_{\calB} \cdot u) / \Boxed{\sim_\calA}\\
    &= (H(w) \cdot HE(u) \cdot \eta_{\calA}) / \Boxed{\sim_\calA}\\
    &= H(w) \cdot (HE(u) / \Boxed{\sim_{HE(\calA)}}) \cdot \eta_{\calA}
  \end{align*}
  because $\eta : \ID_\CC \to HE$ is natural.
  On the other hand, \eqref{gaif.eq.RP2} implies
  $$
    H(w) \cdot \binom{HE(\calB)}{HE(\calA)} \subseteq H(\epsilon_{\calC}^{-1}) \cdot \calM_i^{HE},
  $$
  where
  \begin{align*}
    \calM_i^{HE}
    &= \{ HE(f) / \Boxed{\sim_{HE(\calA)}} : E(f) / \Boxed{\sim_{E(\calA)}} \in \calM_i^E\}\\
    &= \{ HE(f) / \Boxed{\sim_{HE(\calA)}} : f / \Boxed{\sim_{\calA}} \in \calM_i\}.
  \end{align*}
  So, there is an $m / \Boxed{\sim_{\calA}} \in \calM_i$ such that
  \begin{align*}
    w^* \cdot (u / \Boxed{\sim_{\calA}})
    &= H(w) \cdot (HE(u) / \Boxed{\sim_{HE(\calA)}}) \cdot \eta_{\calA}\\
    &= H(\epsilon_{\calC}^{-1}) \cdot (HE(m) / \Boxed{\sim_{HE(\calA)}}) \cdot \eta_{\calA}\\
    &= (H(\epsilon_{\calC}^{-1}) \cdot HE(m) \cdot \eta_{\calA}) / \Boxed{\sim_{\calA}}.
  \end{align*}
  In order to complete the proof of \eqref{gaif.eq.RP3} it suffices to note that
  $H(\epsilon_{\calC}^{-1}) \cdot HE(m) \cdot \eta_{\calA} = m$ because every
  dual equivalence is a special dual adjunction.
  Therefore, $w^* \cdot (u / \Boxed{\sim_{\calA}}) = m / \Boxed{\sim_{\calA}} \in \calM_i$.
\end{proof}

Let us now show that this statement is a proper generalization of \cite[Proposition 9.1 (i)]{KPT}.

\begin{EX}\label{cerp.ex.fba}
    The category $\FinSetInj$ of finite sets and injective maps is dually equivalent to the category
    $\FinBaSurj$ of finite boolean algebras and surjective homomorphisms (Stone duality).
    Since $\FinSetInj$ has the Ramsey property for objects (Example~\ref{cerp.ex.FRP}), it follows that
    the category $\FinBaSurj$ has the dual Ramsey property for objects.

    Let us make this statement explicit. Let $\CC = \FinBaSurj$. For $\calA, \calB \in \Ob(\CC)$
    let $\Surj(\calB, \calA)$ denote the set of all surjective homomorphisms $\calB \twoheadrightarrow \calA$.
    Define $\equiv_\calA$ on $\Surj(\calB, \calA)$ as follows: for $f, f' \in \Surj(\calB, \calA)$ we let
    $f \equiv_\calA f'$ if $f' = \alpha \circ f$ for some $\alpha \in \Aut(\calA)$.

    As in the Example~\ref{cerp.ex.FDRT},
    the fact that $\CC^\op$ has the Ramsey property for objects takes the following form:
    for every integer $k \ge 2$ and all finite boolean algebras $\calA$ and $\calB$
    such that $\Surj(\calB, \calA) \ne \0$ there is a finite boolean algebra
    $\calC$ such that for every $k$-coloring
    $$
      \Surj(\calC, \calA) / \Boxed{\equiv_\calA} \; = \; \calM_1 \union \ldots \union \calM_k
    $$
    there is an $i \in \{1, \ldots, k\}$ and a surjective homomorphism
    $w \in \Surj(\calC, \calB)$ satisfying
    $(\Surj(\calB, \calA) / \Boxed{\equiv_\calA}) \circ w \subseteq \calM_i$.
    Since $\Surj(\calB, \calA) / \Boxed{\equiv_\calA}$ corresponds to congruences $\Phi$ of $\calB$
    such that $\calB / \Phi \cong \calA$ the above statement
    can be reformulated as follows:
    \begin{quote}
      Let $\Con(\calB)$ denote the set of congruences of an algebra $\calB$, and
      for algebras $\calA$ and $\calB$ of the same type let
      $$
        \Con(\calB, \calA) = \{ \Phi \in \Con(\calB) : \calB / \Phi \cong \calA \}.
      $$
      For every finite bolean algebra $\calB$, every $\Phi \in \Con(\calB)$ and
      every $k \ge 2$ there is a finite boolean algebra $\calC$ such that for every
      $k$-coloring of $\Con(\calC, \calB / \Phi)$ there is a congruence
      $\Psi \in \Con(\calC, \calB)$ such that the set of all the congruences from
      $\Con(\calC, \calB / \Phi)$ which contain $\Psi$ is monochromatic.
    \end{quote}
\end{EX}

\begin{EX}\label{cerp.ex.hu}
  By Hu's theorem \cite{hu1,hu2}, every variety generated by a primal algebra is
  categorically equivalent to the variety of boolean algebras. In particular, the
  category $\BA$ whose objects are finite boolean algebras and morphisms are embeddings
  is equivalent to the category $\VV(\calQ)$ whose objects are
  finite algebras in the variety generated by a primal algebra $\calQ$
  and morphisms are embeddings. (Recall that every primal algebra is finite.)
  Therefore, Theorem~\ref{cerp.thm.dual} and Example~\ref{cerp.ex.fba}
  imply that the category $\VV(\calQ)$
  has the Ramsey property for objects for every primal algebra~$\calQ$.
  In other words, we have the following \textit{Ramsey theorem for finite algebras in the variety generated
  by a primal algebra}:
  \begin{quote}
    For every primal algebra $\calQ$,
    for all $\calA, \calB \in \VV(\calQ)$ such that $\calA \hookrightarrow \calB$
    and every $k \ge 2$ there is a $\calC \in \VV(\calQ)$ such that $\calC \longrightarrow (\calB)^{\calA}_k$.
  \end{quote}
  We treat this topic in more detail in Sections~\ref{cerp.sec.primal}, \ref{cerp.sec.fraisse} and~\ref{cerp.sec.oetd}.
\end{EX}

\begin{EX}
  Let $\CC_1$ be the category whose objects are finitely dimensional vector spaces over a fixed finite field $F$
  and whose morphisms are injective linear maps, and
  let $\CC_2$ be the category having the same objects and whose morphisms are surjective linear maps.
  We additionally assume that all the vector spaces in $\CC_1$ and $\CC_2$
  are endowed with the standard inner product $\langle \tilde x, \tilde y\rangle = \sum_i x_i y_i$.
  We know from college algebra that $\CC_1$ is dually equivalent with $\CC_2$ via the functor which
  takes a vector space $V$ to itself and takes a linear map $T$ to its adjoint map $T^*$.
  The Ramsey Theorem for Vector Spaces due to Graham, Leeb and Rotschild (see~\cite{GRS})
  says that the category $\CC_1$ has the Ramsey property for objects.
  Therefore, Theorem~\ref{cerp.thm.dual}~$(c)$ implies that the category
  $\CC_2$ has the dual Ramsey property for objects. In other words, let $F$ be a finite field, and
  for a finite vector space $V$ over $F$ let
  ${\quotient Vd}_{\mathrm{lin}}$ be the set of all partitions of $V$ of the form $V/W = \{v + W : v \in V\}$ where
  $W$ is a subspace of~$V$ of codimension $d$ (so that $\dim(V/W) = d$). Then we have the following \textit{dual Ramsey theorem for
  finite vector spaces}:
  \begin{quote}
    For all positive integers $k$, $a$, $m$ there is a positive integer $n$ such that
    for every $n$-dimensional vector space $V$ over $F$
    and every $k$-coloring of the set ${\quotient Va}_{\mathrm{lin}}$
    there is a partition $\beta \in {\quotient Vm}_{\mathrm{lin}}$ such that
    the set of all patitions from ${\quotient Va}_{\mathrm{lin}}$
    which are coarser than $\beta$ is monochromatic.
  \end{quote}
\end{EX}

\begin{EX}
  The category of finite distributive lattices with lattice embeddings
  does not have the Ramsey property (see~\cite{Promel-Voigt-2}).
  It is quite common that after expanding the structures with appropriatelly chosen linear orders,
  the resulting class of expanded structures has the Ramsey property. Moreover,
  in~\cite{kechris-sokic} the authors prove that
  no expansion of the class of finite distributive lattices by linear orders satisfies the
  Ramsey property.

  Since the class of finite posets is dually equivalent to the class of finite
  distributive lattices (Birkhoff duality), and since the category
  of finite posets and poset embeddings does not have the
  Ramsey property, it follows by Theorem~\ref{cerp.thm.dual} that
  the class of finite distributive lattices and surjective homomorphisms
  does not have the dual Ramsey property either.
  However, it is possible to derive a dual Ramsey
  theorem for finite distributive lattices endowed with a particular linear order which we refer to as
  the \emph{natural order}. For details see~\cite{masul-mudri}.

  On the other hand, since the category of finite distributive lattices with lattice embeddings
  does not have the Ramsey property it follows that the category of finite posets and surjective
  homomorphisms does not have the dual Ramsey property.
\end{EX}

\section{Primal algebras}
\label{cerp.sec.primal}

Let $\calB$ be a finite boolean algebra and let $A = \{ a_1, a_2, \ldots, a_n \}$ be the set of atoms of~$\calB$.
Every linear order $<$ on $A$, say $a_{i_1} < a_{i_2} < \ldots < a_{i_n}$, induces a linear order on $\calB$ as follows.
Take any $x, y \in B$, and let
$x = \delta_1 \cdot a_{i_1} \lor \delta_2 \cdot a_{i_2} \lor \ldots \lor \delta_n \cdot a_{i_n}$ and
$y = \epsilon_1 \cdot a_{i_1} \lor \epsilon_2 \cdot a_{i_2} \lor \ldots \lor \epsilon_n \cdot a_{i_n}$
be the representations of $x$ and $y$, respectively, where $\epsilon_s, \delta_s \in \{0, 1\}$ and with the convention that
$0 \cdot b = 0$ while $1 \cdot b = b$, $b \in B$. We then say that $x \mathrel{\sqsubset} y$ if
there is an $s$ such that $\delta_s < \epsilon_s$, and $\delta_t = \epsilon_t$ for all $t > s$.
In other words, $\sqsubset$ is the \emph{antilexicographic ordering of the elements of $B$} with respect to~$<$.
The choice of the antilexicographic ordering induced by $<$ is motivated by the fact that
the antilexicographic ordering of a bollean algebra is a linear ordering on the algebra that extends
the initial ordering on the atoms (that is, $a_i < a_j$ implies $a_i \sqsubset a_j$).
A linear ordering $\sqsubset$ of a finite boolean algebra $\calB$ is \emph{natural}~\cite{KPT} if there is a linear ordering $<$
on atoms of the algebra such that $\sqsubset$ is the antilexicographic ordering of the elements of $B$ with respect to~$<$.
Let $\OBA$ denote the category whose objects are finite boolean algebras together with a natural linear order
and morphisms are embeddings.

This notion easily generalizes to arbitrary powers of finite algebras.
Let $\calA$ be a finite algebra and let $<$ be an arbitrary linear order on $A$. For every $n \in \NN$
this linear order induces the antilexicographic order $\sqsubset$ on $A^n$ in the usual sense:
$(x_1, \ldots, x_n) \sqsubset (y_1, \ldots, y_n)$ if there is an $s$ such that
$x_i = y_i$ for $i > s$ and $x_s < y_s$.
For every permutation $\pi$ of $\{1, 2, \ldots, n\}$ we also have a linear order $\sqsubset_\pi$ defined by
$(x_1, \ldots, x_n) \sqsubset_\pi (y_1, \ldots, y_n)$ if $(x_{\pi(1)}, \ldots, x_{\pi(n)}) \sqsubset (y_{\pi(1)}, \ldots, y_{\pi(n)})$.
Let $\sqsubseteq$ and $\sqsubseteq_\pi$ denote the reflexive versions of $\sqsubset$ and $\sqsubset_\pi$, respectively.

Let $\calA$ be a primal algebra. (Recall that every primal algebra is finite and has at least two elements.)
It is a well-known fact (see \cite{burris-sankappanavar,jezek} for details on the structure of a variety of algebras
generated by a primal algebra) that if $\calA$ is a primal algebra and $n, m \in \NN$, a mapping $f : A^n \to A^m$ is a homomorphism
from $\calA^n$ to $\calA^m$ if and only if there exist $i_1, \ldots, i_m \in \{1, \ldots, n\}$ such that
$f(x_1, \ldots, x_n) = (x_{i_1}, \ldots, x_{i_m})$. Moreover, $f$ is an embedding if and only if $f$ is injective
if and only if $\{i_1, \ldots, i_m\} = \{1, \ldots, n\}$.

\begin{LEM}\label{cerp.lem.mor}
  Let $\calA$ be a primal algebra, let $<$ be a linear order on~$A$ and let $\sqsubset$ be the induced antilexicographic order.
  Take any $n, m \in \NN$, any permutation $\pi$ of $\{1, \ldots, n\}$ and any permutation $\sigma$ of $\{1, \ldots, m\}$.
  The mapping $f : A^n \to A^m$ is a homomorphism from
  $\calA^n_{\Boxed{\sqsubseteq_\pi}}$ to $\calA^m_{\Boxed{\sqsubseteq_\sigma}}$
  if and only if there exist $i_1, \ldots, i_m \in \{1, \ldots, n\}$ such that $f(x_1, \ldots, x_n) = (x_{i_1}, \ldots, x_{i_m})$
  and the numbers $j_s = \pi^{-1}(i_{\sigma(s)})$, $s \in \{1, \ldots, m\}$, have the following properties:
  \begin{enumerate}
  \item
    $j_m = n$, and
  \item
    for all $s < m$, if $j_s = k < n$ then $\{k+1, \ldots, n\} \subseteq \{j_{s+1}, \ldots, j_{m}\}$.
  \end{enumerate}
\end{LEM}
\begin{proof}
  Note, first, that (ii) is equivalent to the following requirement: if $j_s = k$ is the last appearance of $k$ in the sequence
  $(j_1, j_2, \ldots, j_m)$ then $\{j_{s+1}, \ldots, j_{m}\} = \{k+1, \ldots, n\}$.
  Note, also, that $\{j_1, \ldots, j_m\} = \{d, \ldots, n\}$ where $d = \min\{j_1, \ldots, j_m\}$.

  $(\Rightarrow)$
  Since $f$ is a homomorphism from $\calA^n_{\Boxed{\sqsubseteq_\pi}}$ to $\calA^m_{\Boxed{\sqsubseteq_\sigma}}$
  we know that there exist $i_1, \ldots, i_m \in \{1, \ldots, n\}$ such that $f(x_1, \ldots, x_n) = (x_{i_1}, \ldots, x_{i_m})$,
  and that
  $
    (x_{\pi(1)}, \ldots, x_{\pi(n)}) \sqsubseteq (y_{\pi(1)}, \ldots, y_{\pi(n)})
  $
  implies
  $
    (x_{i_{\sigma(1)}}, \ldots, x_{i_{\sigma(m)}}) \sqsubseteq (y_{i_{\sigma(1)}}, \ldots, y_{i_{\sigma(m)}})
  $. Let $j_s = \pi^{-1}(i_{\sigma(s)})$, $s \in \{1, \ldots, m\}$, so that $i_{\sigma(s)} = \pi(j_s)$ for all~$s$.

  Let us show that $j_m = n$, that is, $i_{\sigma(m)} = \pi(n)$. Suppose this is not the case and let
  $i_{\sigma(m)} = \pi(k)$ for some $k < n$. Take any $a, b \in A$ so that $a < b$ and consider the $n$-tuples
  \begin{align*}
    \overline x = (x_{\pi(1)}, \ldots, x_{\pi(k)}, \ldots, x_{\pi(n)}) &= (a, a, \ldots, a, \place{k}{th}{b}, a, \ldots, a),\\
    \overline y = (y_{\pi(1)}, \ldots, y_{\pi(k)}, \ldots, y_{\pi(n)}) &= (a, a, \ldots, a, \place{k}{th}{a}, a, \ldots, b).
  \end{align*}
  Then
  $
    \overline x \sqsubset \overline y
  $
  but
  $
    (x_{i_{\sigma(1)}}, \ldots, x_{i_{\sigma(m)}}) \sqsupset (y_{i_{\sigma(1)}}, \ldots, y_{i_{\sigma(m)}})
  $ as $x_{i_{\sigma(m)}} = x_{\pi(k)} = b > a = y_{\pi(k)} = y_{i_{\sigma(m)}}$. Contradiction.

  Let us now show that (ii) holds for the sequence $(j_1, \ldots, j_m)$. Suppose, to the contrary, that there is an
  $s < m$ such that $j_s = k < n$ but $\{k+1, \ldots, n\} \not\subseteq \{j_{s+1}, \ldots, j_{m}\}$.
  Take the largest $l \in \{k+1, \ldots, n\} \setminus \{j_{s+1}, \ldots, j_{m}\}$. Note that $l \le n-1$ as $j_m = n$.
  Take any $a, b \in A$ so that $a < b$ and consider the $n$-tuples
  \begin{align*}
    \overline x = (x_{\pi(1)}, \ldots, x_{\pi(k)}, \ldots, x_{\pi(n)}) &= (a, \ldots, a, \place{k}{th}{b}, a, \ldots, a, \place{l}{th}{a}, a, \ldots, a),\\
    \overline y = (y_{\pi(1)}, \ldots, y_{\pi(k)}, \ldots, y_{\pi(n)}) &= (a, \ldots, a, \place{k}{th}{a}, a, \ldots, a, \place{l}{th}{b}, a, \ldots, a).
  \end{align*}
  Then
  $
    \overline x \sqsubset \overline y
  $
  but, having in mind that $i_{\sigma(s)} = \pi(j_s) = \pi(k)$,
  \begin{align*}
    (x_{i_{\sigma(1)}}, \ldots, x_{i_{\sigma(s)}}, \ldots, x_{i_{\sigma(m)}})
    & = (\ldots, \place{s}{th}{b}, \underbrace{a, \ldots, a, a, \ldots, a}_{\text{no index equals $\pi(l)$}})\\
    &\sqsupset (\ldots, \place{s}{th}{a}, \underbrace{a, \ldots, a, a, \ldots, a}_{\text{no index equals $\pi(l)$}})\\
    &= (y_{i_{\sigma(1)}}, \ldots, y_{i_{\sigma(s)}}, \ldots, y_{i_{\sigma(m)}}).
  \end{align*}
  Contradiction.

  $(\Leftarrow)$
  Let $f : A^n \to A^m$ be a mapping such that $f(x_1, \ldots, x_n) = (x_{i_1}, \ldots, x_{i_m})$ for some
  $i_1, \ldots, i_m \in \{1, \ldots, n\}$ and assume that the numbers
  $j_s = \pi^{-1}(i_{\sigma(s)})$, $s \in \{1, \ldots, m\}$, satisfy (i) and (ii). Then $f$ is clearly a homomorphism
  from $\calA^n$ to $\calA^m$, so let us show that $f$ is monotonous. Take $x_1, \ldots, x_n, y_1, \ldots, y_n \in A$
  such that $(x_{\pi(1)}, \ldots, x_{\pi(k)}, \ldots, x_{\pi(n)})  \sqsubset (y_{\pi(1)}, \ldots, y_{\pi(k)}, \ldots, y_{\pi(n)})$.
  Then there is a $t$ such that $x_{\pi(q)} = y_{\pi(q)}$ for all $q > t$ and $x_{\pi(t)} < y_{\pi(t)}$.
  Let us show that
  \begin{equation}\label{eq.hom}
    (x_{\pi(j_1)}, \ldots, x_{\pi(j_m)}) \sqsubseteq (y_{\pi(j_1)}, \ldots, y_{\pi(j_m)}).
  \end{equation}
  If $\min\{j_1, \ldots, j_m\} > t$ then equality holds in $\eqref{eq.hom}$.
  Suppose, therefore, that $\min\{j_1, \ldots, j_m\} \le t$.
  Then $t \in \{j_1, \ldots, j_m\}$ because of~(ii). Let $j_s$ be the
  last appearance of $t$ in the sequence $(j_1, \ldots, j_m)$. Then $\{j_{s+1}, \ldots, j_{m}\} =
  \{t+1, \ldots, n\}$, whence follows that strict inequality holds in~$\eqref{eq.hom}$.

  Therefore, $\eqref{eq.hom}$ holds. The choice of the indices $j_s$ ensures that $\eqref{eq.hom}$ is equivalent to
  $(x_{i_{\sigma(1)}}, \ldots, x_{i_{\sigma(m)}}) \sqsubseteq (y_{i_{\sigma(1)}}, \ldots, y_{i_{\sigma(m)}})$.
\end{proof}

  Let $\calA$ be a primal algebra and let $<$ be a linear order on $A$.
  Let $\OV(\calA, \Boxed<)$ be the category whose objects are isomorphic copies of structures
  $\calA^n_{\Boxed{\sqsubseteq_\pi}}$ where $n \in \NN$ and $\pi$ is a permutation of $\{1, \ldots, n\}$,
  and whose morphisms are embeddings.

\begin{THM}\label{cerp.thm.k-oba}
  Let $\calA$ be a primal algebra and let $<$ be a linear order on $A$. Then
  $\OV(\calA, \Boxed<)$ is categorically equivalent to $\OBA$.
\end{THM}
\begin{proof}
  Every finite boolean algebra together with a natural is clearly isomorphic to
  $\Bii^n_{\Boxed{\sqsubseteq_\pi}}$ where $\Bii$ is the two-element boolean algebra whose base set
  is $2 = \{0, 1\}$,
  $n \in \NN$, and $\pi$ is a permutation of $\{1, \ldots, n\}$ which encodes the initial ordering of the atoms.
  Hence, $\OBA = \OV(\Bii, \Boxed\prec)$, where $\prec$ is the usual ordering $0 \prec 1$ of $\Bii$.

  Let $\BB$ be the full subcategory of $\OBA$ spanned by the countable set of objects $\big\{\Bii^n_{\Boxed{\sqsubseteq_\pi}} : n \in \NN$,
  $\pi$ is a permutation of $\{1, 2,\ldots, n\}\big\}$, and let $\CC$ be the full subcategory of $\OV(\calA, \Boxed<)$ spanned by
  the countable set of objects $\big\{\calA^n_{\Boxed{\sqsubseteq_\pi}} : n \in \NN$, $\pi$ is a permutation of $\{1, 2,\ldots, n\}\big\}$.
  Clearly, $\BB$ and $\CC$ are skeletons of $\OBA$ and $\OV(\calA, \Boxed<)$, respectively, so in order to show that
  $\OBA$ and $\OV(\calA, \Boxed<)$ are equivalent it suffices to show that $\BB$ and $\CC$ are isomorphic.
  But this is easy!
  Let $F : \BB \to \CC$ be a functor such that $F(\Bii^n_{\Boxed{\sqsubseteq_\pi}}) = \calA^n_{\Boxed{\sqsubseteq_\pi}}$
  and which takes a morphism $f : 2^n \to 2^m : (x_1, \ldots, x_n) \mapsto (x_{i_1}, \ldots, x_{i_m})$ to
  $f' : A^n \to A^m : (x_1, \ldots, x_n) \mapsto (x_{i_1}, \ldots, x_{i_m})$. Lemma~\ref{cerp.lem.mor} ensures that $F$ is
  well defined and bijective on morphisms, so $F$ is clearly an isomorphism of $\BB$ and $\CC$. Therefore,
  the categories $\OBA$ and $\OV(\calA, \Boxed<)$ are equivalent.
\end{proof}

\section{\Fraisse\ classes and categorical equivalence}
\label{cerp.sec.fraisse}

For a countable structure $\calM$, the class of all finitely generated substructures of $\calM$
is called the \emph{age} of $\calM$ and we denote it by~$\age(\calM)$. A class $\KK$ of finite
structures is an \emph{age} if there is countable structure $\calM$ such that
$\KK = \age(\calM)$.
It is a well-known result that a class $\KK$ of finite structures is an age if and only if
\begin{itemize}
\item
  $\KK$ is an abstract class (that is, closed for isomorphisms),
\item
  there are at most countably many pairwise nonisomorphic structures in $\KK$,
\item
  $\KK$ has the hereditary property (HP), and
\item
  $\KK$ has the joint embedding property (JEP).
\end{itemize}
An age $\KK$ is a \emph{\Fraisse\ age} (= \Fraisse\ class = amalgamation class) if $\KK$ satisfies the
amalgamation property (AP).

A countable structure $\calM$ is ultrahomogeneous if partial isomorphisms between finite substructures lift to an automorphism of the entire structure.  In other words, for any tuple $\oa$ from $\calM$, the orbit of $\oa$ under $\Aut(\calM)$ is defined by the quantifier-free type of $\oa$.  An $\LW$ formula is a formula built out of basic relations, countable conjunctions/disjunctions and negations (called a \emph{simple} formula in \cite{KPT}.)   In general the orbit of a finite tuple is defined by a quantifier-free $\LW$-formula, using Scott sentences.

For every \Fraisse\ age $\KK$ there is a unique (up to isomorphism) countable ultrahomogeneous structure $\calA$
such that $\KK = \age(\calA)$. We say that $\calA$ is the \emph{\Fraisse\ limit} of $\KK$, denoted $\Flim \KK$.  For further model theoretic background, see \cite{hodges}.

If $\KK$ is a Ramsey class of finite ordered structures which is closed under
isomorphisms and taking substructures, and has the joint embedding property, then $\KK$
is a \Fraisse\ age~\cite{Nesetril}. In that case we say that $\KK$ is a \emph{Ramsey age}.
So, every Ramsey age is a \Fraisse\ age.

\begin{LEM}
  Let $\CC$ and $\DD$ be equivalent categories whose objects are structures and embeddings are morphisms.

  $(a)$ If one of the two categories has (JEP) then so does the other.

  $(b)$ If one of the two categories has (AP) then so does the other.
\end{LEM}
\begin{proof}
  Let us show $(b)$ (the proof of $(a)$ is similar).

  Let $E : \CC \to \DD$ and $H : \DD \to \CC$ be functors that constitute the equivalence between $\CC$ and $\DD$
  and let $\eta : \ID_\CC \to HE$ and $\epsilon : \ID_\DD \to EH$ be the accompanying natural isomorphisms.
  Assume that $\DD$ has (AP) and let $f : \calA \hookrightarrow \calB$ and $g : \calA \hookrightarrow \calC$ be two
  embeddings in $\CC$. Then $E(f) : E(\calA) \hookrightarrow E(\calB)$ and $E(g) : E(\calA) \hookrightarrow E(\calC)$ are
  embeddings in $\DD$, so there is a $\calD \in \Ob(\DD)$ and embeddings $u : E(\calB) \hookrightarrow \calD$ and
  $v : E(\calC) \hookrightarrow \calD$ such that $u \circ E(f) = v \circ E(g)$. Then the diagram
  $$
    \XYMATRIX@M=8pt{
      & H(\calD) \arembfrom[dr]^-{H(u)} & \\
      HE(\calC) \aremb[ur]^-{H(v)} \arembfrom[r]_-{HE(g)} & HE(\calA) \aremb[r]_-{HE(f)} & HE(\calB) \\
      \calC \ar[u]^-{\eta_\calC} \arembfrom[r]^-g & \calA \ar[u]^-{\eta_\calA} \aremb[r]^-{f} & \calB \ar[u]_-{\eta_\calB}
    }
  $$
  commutes because $H$ is a functor and because $\eta$ is natural.
\end{proof}

\begin{COR}\label{cerp.cor.rce}
  Let $\CC$ and $\DD$ be equivalent categories with finite structures as objects and embeddings as morphisms.
  If $\CC$ is a Ramsey age and $\DD$ has (HP) then $\DD$ is also a Ramsey age.
\end{COR}
\begin{proof}
  Follows from the above lemma and Theorem~\ref{cerp.thm.dual}.
\end{proof}

\begin{LEM}
  Let $\calA$ be a primal algebra and let $<$ be a linear order on $A$. Then
  $\OV(\calA, \Boxed<)$ has (HP).
\end{LEM}
\begin{proof}
  Take any $\calA^m_{\sqsubseteq_\sigma}$ and any embedding $f : \calA^n_\preccurlyeq \hookrightarrow \calA^m_{\sqsubseteq_\sigma}$.
  Let $i_1, \ldots, i_m$ be indices such that $f(x_1, \ldots, x_n) = (x_{i_1}, \ldots, x_{i_m})$ and $\{i_1, \ldots, i_m\} = \{1, \ldots, n\}$.
  Therefore, $(x_1, \ldots, x_n) \preccurlyeq (y_1, \ldots, y_n)$ if and only if
  $(x_{i_1}, \ldots, x_{i_m}) \sqsubseteq_\sigma (y_{i_1}, \ldots, y_{i_m})$.

  Let us show that there exists a permutation $\pi$ of $\{1, \ldots, n\}$ such that $\Boxed\preccurlyeq = \Boxed{\sqsubseteq_\pi}$.
  For an arbitrary $s \in \{1, \ldots, n\}$ let $A_s = \{ t \in \{1, \ldots, m\} : i_{\sigma(t)} = s \}$. Note that
  $\{A_1, \ldots, A_n\}$ is a partition of $\{1, \ldots, m\}$ because $\{i_1, \ldots, i_m\} = \{1, \ldots, n\}$.
  Let $\pi$ be a permutation of $\{1, \ldots, n\}$ such that $\max A_{\pi(1)} < \max A_{\pi(2)} < \ldots <  \max A_{\pi(n)}$
  (note that $m \in A_{\pi(n)}$), and define $(j_1, \ldots, j_m)$ as follows: $j_s = k$ if and only if $s \in A_{\pi(k)}$.
  Then it is easy to verify that $\pi(j_s) = i_{\sigma(s)}$ for all $s$ and that $(j_1, \ldots, j_m)$
  satisfies (i) and (ii) of Lemma~\ref{cerp.lem.mor}. So, Lemma~\ref{cerp.lem.mor} ensures that
  $f$ is a homomorphism, and hence an embedding, of $\calA^n_{\sqsubseteq_\pi}$ into $\calA^m_{\sqsubseteq_\sigma}$.

  Let us show that $\Boxed\preccurlyeq = \Boxed{\sqsubseteq_\pi}$.
  One the one hand, $(x_1, \ldots, x_n) \sqsubseteq_\pi (y_1, \ldots, y_n)$ is equivalent to
  $(x_{i_1}, \ldots, x_{i_m}) \sqsubseteq_\sigma (y_{i_1}, \ldots, y_{i_m})$ because
  $f : \calA^n_{\sqsubseteq_\pi} \hookrightarrow \calA^m_{\sqsubseteq_\sigma}$.
  On the other hand, $(x_1, \ldots, x_n) \preccurlyeq (y_1, \ldots, y_n)$ is equivalent to
  $(x_{i_1}, \ldots, x_{i_m}) \sqsubseteq_\sigma (y_{i_1}, \ldots, y_{i_m})$ because
  $f : \calA^n_\preccurlyeq \hookrightarrow \calA^m_{\sqsubseteq_\sigma}$. Therefore,
  $\Boxed\preccurlyeq = \Boxed{\sqsubseteq_\pi}$.
\end{proof}

\begin{THM}\label{cerp.thm.1}
  Let $\calA$ be a primal algebra and let $<$ be a linear order on $A$. Then
  $\OV(\calA, \Boxed<)$ is a Ramsey age.
\end{THM}
\begin{proof}
  Since $\OV(\calA, \Boxed<)$ is categorically equivalent to $\OBA$ (Theorem~\ref{cerp.thm.k-oba}),
  $\OBA$ is a Ramsey age \cite{KPT} and $\OV(\calA, \Boxed<)$ has (HP) by above lemma,
  Corollary~\ref{cerp.cor.rce} yields that $\OV(\calA, \Boxed<)$ is a Ramsey age.
\end{proof}

\section{\Fraisse\ limits with identical automorphism group}
\label{cerp.sec.Lynn}

%%Lynn: 18 Nov, small clarification
For a closed subgroup $G < \Sym(\mathbb{N})$ we can construct an ultrahomogeneous structure as follows.  Let $\Delta_{h(G)} = \{R_i \mid i \in I\}$ where $R_i$ is an $n$-ary relation corresponding to the orbit $\calO_i$ of an $n$-tuple $\oa \in \mathbb{N}^n$ under $G$.   Define $\calM = (\mathbb{N}, \Delta_{h(G)})$ so that
$$\calM \vDash R_i(\oa) \Leftrightarrow \oa \in \calO_i$$
Clearly partial isomorphisms of $\calM$ extend to automorphisms of $\calM$ and so $\calM$ is ultrahomogeneous.  Moreover, $\Aut(\calM) = G$.  We will call $\Delta_{h(G)}$ the \emph{Hodges language corresponding to $G$} ($(\mathbb{N}, \Delta_{h(G)})$ is called the ``induced structure associated to $G$'' in \cite{KPT}).
In the case that $G = \Aut(\calA)$ for some given countable structure $\calA$, we call $\Delta_{h(G)}$ the \emph{Hodges language on $\calA$}.
%%Lynn: end

In \cite{KPT} the authors show the following.

\begin{THM} \cite[Theorem 4.7]{KPT}\label{kpt4.7}
  Let $G$ be a closed subgroup of $\Sym(F)$ for a countable set $F$. Then $G$
  is extremely amenable if and only if $G = \Aut(\calF)$ for a countable homogeneous structure $\calF$
  whose age has the Ramsey property and consists of rigid elements.
\end{THM}

This theorem guarantees that for extremely amenable $G\leq\Sym(\mathbb{N})$, $\calM_G = \{\calM \mid \Aut(\calM)=G \textrm{~and $\calM$ is ordered by $<$ and ultrahomogeneous}\}$ gives a family of structures with age having the Ramsey property for morphisms.  See Appendix 2 for an additional discussion involving function symbols.  The canonical relational structure $(\mathbb{N},\Delta_{h(G)})$ is in $\calM_G$, but so are a variety of structures in functional languages whose description might be more natural.  Consider the family of categories
$$\Gamma = \{\CC \mid \CC \textrm{~is the category of structures with $\Ob(\CC) = \age(\calM)$ for $\calM \in \calM_G$} \}$$
We can use our technology of adjunctions to explain in categorical language why all members $\CC \in \Gamma_G$ must share the Ramsey property for morphisms.

\begin{THM}\label{equiv1}
Let $\DD_1, \DD_2$ be two categories of finite structures.  Suppose each class of structures is a \Fraisse\ class and the automorphism groups of the \Fraisse\ limits are isomorphic.  Then there is a third category of finite structures $\CC$ and adjunctions
$$F_1 : \CC \rightleftarrows \DD_1 : G_1$$
$$F_2 : \CC \rightleftarrows \DD_2 : G_2$$
such that the $G_i$ are inclusions and $F_2 \circ G_1$, $F_1 \circ G_2$ preserve the Ramsey property for morphisms.
\end{THM}

\begin{proof}  Let $\calF_1, \calF_2$ be the \Fraisse\ limits of the classes $\Ob(\DD_1), \Ob(\DD_2)$,
  both with automorphism group $G$.  We may assume that $\calF_1, \calF_2$ have the same underlying set, $\NN$.
  Now let $\Ob(\CC)$ be all finite subsets of $\NN$ and let the morphisms of $\CC$ be embeddings in the
  Hodges language corresponding to $G$.  Now define $F_i : \CC \rightarrow \DD_i$ to take any finite subset
  $A \subset \NN$ and send it to the closure of $A$ under the function symbols in the language of $\calF_i$.
  This map is well-defined and natural by ultrahomogeneity of $\calF_i$.  Thus we have an adjunction
  $F_i : \CC \rightleftarrows \DD_i : G_i$, where $G_i$ is the inclusion functor that ``forgets'' the function symbols.
  By Proposition~\ref{cofinal}, $F_i$ preserves the Ramsey property for morphisms.  By Theorem \ref{cerp.thm.adj}, so does $G_i$.  Thus, so do their composites.
\end{proof}

\begin{OBS} It is interesting to note that in the above Theorem, $F_i \circ G_i = \ID_{\DD_i}$ but $G_i \circ F_i \neq \ID_{\DD_i}$.
\end{OBS}

We thank Christian Rosendal for letting us include the example $\KK_r$ studied in his unpublished notes.  The example $\KK_s$ is from \cite{sh90}.

\begin{DEF} Consider the following classes of trees as categories with embeddings as morphisms.

$$\KK_r = \age(\W, f,<)$$
$$\KK_s = \age(\W, \lhd, \wedge, <, \{P_n\}_n)$$
$$\KK_h = \age(\W, \Delta_{h(\Aut(\Flim \KK_r))}$$

where
\begin{itemize}
\item $\lhd$ is the partial tree order (sequence extension)
\item $\wedge$ is the meet in the partial order
\item $<$ is the lexicographic order on sequences (a linear extension of the partial order)
\item $P_n$ is a unary predicate picking out the $n$th level of the tree
\item $f$ is a unary function symbol giving the immediate $\lhd$-predecessor of any node
\end{itemize}
\end{DEF}

%%Lynn: 18 Nov, clarified
\noindent The following is guaranteed by Theorem \ref{kpt4.7}.

\begin{COR} $\KK_r$ has the Ramsey property for morphisms just in case $\KK_s$ has the Ramsey property for morphisms.
\end{COR}

\begin{proof} The classes are \Fraisse\ classes, so we may assume that $\calF_s = \Flim \KK_s$,
  $\calF_r = \Flim \KK_r$ share the same underlying set $\NN$.
  Every function and predicate symbol in $\KK_r$ is quantifier-free $\LW$-definable in the language of $\KK_s$ and vice-versa. Thus their automorphism groups have the same orbits on $n$-tuples from $\mathbb{N}$, and thus are the same group.
\end{proof}

\begin{OBS}
Substructures of $I_r$ are closed under the function symbols in $I_s$ so we could set up a direct adjunction $F : \KK_s \rightleftarrows \KK_r : G$.  In other words, the intermediary category $\CC$ from Theorem \ref{equiv1} can be jettisoned in favor of the more direct Theorem \ref{cerp.thm.adj}.
\end{OBS}
%%Lynn: end

\section{Order expansions and topological dynamics under categorical equivalence}
\label{cerp.sec.oetd}

Let $\CC$ be a category of finite structures and embeddings, and $\CC^*$ a category of finite ordered
structures and embeddings. We say that $\CC^*$ \emph{is an order expansion of} $\CC$ (cf.~\cite{KPT}) if
\begin{itemize}
\item for every structure $\calA_< = (A, \Delta, \Boxed<) \in \Ob(\CC^*)$ we have
      that $\calA = (A, \Delta) \in \Ob(\CC)$, and
\item the forgetful functor $U : \CC^* \to \CC$ which acts on objects by $U(A, \Delta, \Boxed<) = (A, \Delta)$
      and on morphisms by $U(f) = f$ is surjective on objects.
\end{itemize}

An order expansion $\CC^*$ of $\CC$ is \emph{reasonable} (cf.~\cite{KPT})
if for all $\calA, \calB \in \Ob(\CC)$, every embedding $f : \calA \hookrightarrow \calB$ and every
$\calA_< \in \Ob(\CC^*)$ such that $U(\calA_<) = \calA$ there is a $\calB_\sqsubset \in \Ob(\CC^*)$
such that $U(\calB_\sqsubset) = \calB$ and $f$ is an embedding of $\calA_<$ into $\calB_\sqsubset$.
It is easy to show that if $\CC^*$ is a reasonable expansion of $\CC$ and $\CC^*$ has (HP), resp.\ (JEP) or (AP), then
$\CC$ has (HP), resp.\ (JEP) or (AP) (cf.~\cite{KPT}); consequently if $\Ob(\CC^*)$ is a \Fraisse\ age, then so is $\Ob(\CC)$.

Let $\CC^*$ be an order expansion of $\CC$.
We say that $\CC^*$ has the \emph{ordering property over} $\CC$ if the following holds:
for every $\calA \in \Ob(\CC)$ there is a $\calB \in \Ob(\CC)$ such that $\calA_< \embedsto \calB_\sqsubset$
for all $\calA_<, \calB_\sqsubset \in \Ob(\CC^*)$ such that $U(\calA_<) = \calA$ and $U(\calB_\sqsubset) = \calB$.
We say that $\calB$ is a \emph{witness of the ordering property for $\calA$}.

\begin{LEM}\label{cerp.lem.op}
  Let $\CC^*$ be a reasonable order expansion of $\CC$ with the forgetful functor $U : \CC^* \to \CC : \calA_< \mapsto \calA, f \mapsto f$
  and let $\DD^*$ be a reasonable order expansion of $\DD$ with the forgetful functor $V : \DD^* \to \DD : \calA_< \mapsto \calA, f \mapsto f$.
  Assume that $E^* : \CC^* \rightleftarrows \DD^* : H^*$ is a categorical equivalence of $\CC^*$ and $\DD^*$, that
  $E : \CC \rightleftarrows \DD : H$ is a categorical equivalence of $\CC$ and $\DD$, and that the following diagrams commute:
  $$
    \XYMATRIX{
      \CC^* \ar[r]^-{E^*} \ar[d]_-U & \DD^* \ar[d]^-V & & \CC^* \ar[d]_-U & \DD^* \ar[d]^-V \ar[l]_-{H^*}\\
      \CC   \ar[r]_-{E}             & \DD             & & \CC             & \DD \ar[l]^-{H}
    }
  $$
  Then $\CC^*$ has the ordering property over $\CC$ if and only if $\DD^*$ has the ordering property over $\DD$.
\end{LEM}
\begin{proof}
  Assume that $\CC^*$ has the ordering property over $\CC$ and let us show that $\DD^*$ has the ordering property over $\DD$.
  Take any $\calA \in \Ob(\DD)$. Then $H(\calA) \in \Ob(\CC)$ so, by the ordering property, there is a $\calB \in \Ob(\CC)$
  which is a witness of the ordering property for $H(\calA)$. Let us show that $E(\calB) \in \Ob(\DD)$
  is a witness of the ordering property for $\calA$. Take any $\calA_<, \calB_\sqsubset \in \Ob(\DD^*)$ such that
  $V(\calA_<) = \calA$ and $V(\calB_\sqsubset) = E(\calB)$ and let us show that $\calA_< \hookrightarrow \calB_\sqsubset$.

  Let us first show that $H^*(\calA_<) \hookrightarrow H^*(\calB_\sqsubset)$. Note first that $UH^*(\calA_<) = HV(\calA_<) = H(\calA)$
  and that $UH^*(\calB_\sqsubset) = HV(\calB_\sqsubset) = HE(\calB) \cong \calB$. Since $\CC^*$ is a reasonable order expansion of $\CC$,
  there is a $\calB_\prec \in \Ob(\CC^*)$ such that $\calB_\prec \cong H^*(\calB_\sqsubset)$ and $U(\calB_\prec) = \calB$.
  Since $\CC^*$ has the ordering property over $\CC$ and $\calB$ is a witness of the ordering property for $H(\calA)$,
  we have that $H^*(\calA_<) \hookrightarrow \calB_\prec \cong H^*(\calB_\sqsubset)$. Therefore,
  $\calA_< \cong E^*H^*(\calA_<) \hookrightarrow E^*H^*(\calB_\sqsubset) \cong \calB_\sqsubset$.
\end{proof}

Let $G$ be a topological group. Its \emph{action} on $X$ is a mapping $\Boxed \cdot : G \times X \to X$
such that $1 \cdot x = x$ and $g \cdot (f \cdot x) = (gf) \cdot x$.
We also say that $G$ \emph{acts} on $X$. A \emph{$G$-flow} is a continuous action of a topological group $G$
on a topological space $X$. A \emph{subflow} of a $G$-flow $\Boxed\cdot : G \times X \to X$
is a continuous map $\Boxed* : G \times Y \to Y$ where $Y \subseteq X$ is a closed subspace of $X$ and
$g * y = g \cdot y$ for all $g \in G$ and $y \in Y$.
A $G$-flow $G \times X \to X$ is \emph{minimal} if it has no proper closed subflows.
A $G$-flow $u : G \times X \to X$ is \emph{universal}
if every compact minimal $G$-flow $G \times Z \to Z$ is a factor of~$u$.
It is a well-known fact that for a compact Hausdorff space $X$ there is, up to isomorphism of $G$-flows,
a unique universal minimal $G$-flow, usually denoted by $G \curvearrowright M(G)$.

A topological group $G$ is \emph{extremely amenable}
if every $G$-flow $\Boxed\cdot : G \times X \to X$
on a compact Hausdroff space $X$ has a joint fix point, that is, there is an $x_0 \in X$ such that $g \cdot x_0 = x_0$
for all $g \in G$. Since $\Sym(A)$ carries naturally the topology of pointwise convergence, permutation groups can
be thought of as topological groups. For example, it was shown in~\cite{Pestov-1998} that $\Aut(\QQ, \Boxed<)$ is extremely amenable
while $\Sym(A)$, the group of all permutations on $A$, is not for a countably infinite set $A$.

The following is a corollary of Theorem \ref{kpt4.7}

\begin{COR}
  Let $\calA$ be a primal algebra and let $<$ be a linear order on $A$. Then
  $\OV(\calA, \Boxed<)$ is a Ramsey age (Theorem~\ref{cerp.thm.1}), so the
  automorphism group of its \Fraisse\ limit is extremely amenable.
\end{COR}

Let $\LO(A)$ be the set of all linear orders on $A$ and let $G$ be a closed subgroup of $\Sym(A)$.
The set $\LO(A)$ with the topology of pointwise convergence is a compact Hausdorff
space and the action of $G$ on $\LO(A)$ given by $x \mathbin{<^g} y \text{ if and only if } g^{-1}(x) < g^{-1}(y)$ is continuous.
This action is usually referred to as the \emph{logical action of $G$ on $\LO(A)$}.

\begin{THM} \cite[Theorem 10.8]{KPT}\label{cerp.thm.KPT2}
  Let $\KK^*$ be a \Fraisse\ age which is a reasonable order expansion of a \Fraisse\ age $\KK$.
  Let $\calF$ be the \Fraisse\ limit of $\KK$, let $\calF_\sqsubset$ be the \Fraisse\ limit of $\KK^*$,
  let $G = \Aut(\calF)$ and $X^* = \overline{G \cdot \Boxed\sqsubset}$ (in the logical action of $G$ on $\LO(F)$).
  Then the logical action of $G$ on $X^*$ is the universal minimal flow of $G$ if and only if
  the class $\KK^*$ has the Ramsey property, as well as the ordering property with respect to $\KK$.
\end{THM}

We shall now apply this result to the classes $\OV(\calA, \Boxed<)$ and $\VV(\calA)$, where
$\calA$ is a primal algebra and let $<$ be a linear order on $A$. Let us first show that the former
is a reasonable order expansion of the latter.

\begin{LEM}
  Let $\calA$ be a primal algebra and let $<$ be a linear order on $A$.
  Then $\OV(\calA, \Boxed<)$ is a reasonable order expansion of $\VV(\calA)$.
\end{LEM}
\begin{proof}
  Let $f : \calA^n \hookrightarrow \calA^m$ be an embedding, and let
  $i_1, \ldots, i_m$ be indices such that $f(x_1, \ldots, x_n) = (x_{i_1}, \ldots, x_{i_m})$
  and $\{i_1, \ldots, i_m\} = \{1, \ldots, n\}$. Take any permutation
  $\pi$ of $\{1, 2, \ldots, n\}$ and let us find a permutation $\sigma$ of $\{1, 2, \ldots, m\}$
  such that $f$ is an embedding of $\calA^n_{\sqsubseteq_\pi}$ into $\calA^m_{\sqsubseteq_\sigma}$.
  For an arbitrary $s \in \{1, \ldots, n\}$ let $A_s = \{ t \in \{1, \ldots, m\} : \pi(s) = i_t \}$. Note that
  $\{A_1, \ldots, A_n\}$ is a partition of $\{1, \ldots, m\}$ because $\{i_1, \ldots, i_m\} = \{1, \ldots, n\}$.

  Let $\sigma$ be any permutation of $\{1, \ldots, m\}$ such that $\sigma^{-1}(A_1) = \{1, \ldots, k_1\}$,
  $\sigma^{-1}(A_2) = \{k_1+1, \ldots, k_2\}$, \ldots, $\sigma^{-1}(A_n) = \{k_{n-1}+1, \ldots, m\}$ and define
  $(j_1, \ldots, j_m)$ as follows: $j_1 = \ldots = j_{k_1} = 1$,
  $j_{k_1+1} = \ldots = j_{k_2} = 2$, \ldots, $j_{k_{n-1}+1} = \ldots = j_{m} = n$.
  Then it is easy to verify that $\pi(j_s) = i_{\sigma(s)}$ for all $s$ and that $(j_1, \ldots, j_m)$
  satisfies (i) and (ii) of Lemma~\ref{cerp.lem.mor}. Now, Lemma~\ref{cerp.lem.mor} ensures that
  $f$ is a homomorphism, and hence an embedding, of $\calA^n_{\sqsubseteq_\pi}$ into $\calA^m_{\sqsubseteq_\sigma}$.
\end{proof}

Let $\calA$ be a primal algebra and let $<$ be a linear order on $A$.
The class $\OV(\calA, \Boxed<)$ is a Ramsey age by Theorem~\ref{cerp.thm.1}, while Lemma~\ref{cerp.lem.op} ensures that
$\OV(\calA, \Boxed<)$ has the ordering property over $\VV(\calA)$. (The categorical equivalences in question are
$\BA \rightleftarrows \VV(\calA)$ and $\OBA \rightleftarrows \OV(\calA, \Boxed<)$ established in Example~\ref{cerp.ex.hu}
and Theorem~\ref{cerp.thm.k-oba}.) Then Theorem~\ref{cerp.thm.KPT2} yields:

\begin{COR}
  Let $\calA$ be a primal algebra and let $<$ be a linear order on $A$.
  Let $\calF$ be the \Fraisse\ limit of $\VV(\calA)$, let $\calF_\sqsubset$ be the \Fraisse\ limit of $\OV(\calA, \Boxed<)$,
  let $G = \Aut(\calF)$ and $X^* = \overline{G \cdot \Boxed\sqsubset}$ (in the logical action of $G$ on $\LO(F)$).
  Then the logical action of $G$ on $X^*$ is the universal minimal flow of $G$.
\end{COR}

\section{Appendix 1: The product Ramsey theorem revisited}
\label{cerp.sec.prod}

In this section we generalize the product Ramsey theorem of M.~Soki\'c~\cite{sokic2,sokic-phd}.
We provide an abstract proof that if two categories have
some of Ramsey property, then their categorical product has the ``combined'' kind of Ramsey
property. As a consequence we have that for every category $\CC$ and every positive integer $n$,
if $\CC$ a (dual) Ramsey property then so does $\CC^n$, showing thus a
metaresult that every finite (dual) Ramsey theorem has the finite product version.

For categories $\CC_1$ and $\CC_2$ there is a category $\CC_1 \times \CC_2$ whose objects are pairs $(\calA_1, \calA_2)$
where $\calA_1 \in \Ob(\CC_1)$ and $\calA_2 \in \Ob(\CC_2)$,
morphisms are pairs $(f_1, f_2) : (\calA_1, \calA_2) \to (\calB_1, \calB_2)$ where
$f_i$ is a morphism from $\calA_i$ to $\calB_i$ in $\CC_i$, $i \in \{1, 2\}$, and the composition of morphisms
is carried out componentwise: $(f_1, f_2) \cdot (g_1, g_2) = (f_1 \cdot g_1, f_2 \cdot g_2)$.
Clearly, if $\tilde\calA = (\calA_1, \calA_2)$ and $\tilde\calB = (\calB_1, \calB_2)$ are objects of $\CC_1 \times \CC_2$
then
$$
  \hom(\tilde\calA, \tilde\calB) = \hom(\calA_1, \calB_1) \times \hom(\calA_2, \calB_2)
$$
and
$$
  \binom{\tilde\calB}{\tilde\calA} = \binom{\calB_1}{\calA_1} \times \binom{\calB_2}{\calA_2}.
$$
By iterating this construction we now see that
for each category $\CC$ and each $n \in \NN$ there is a category $\CC^n$ whose objects are tuples $(\calA_1, \ldots, \calA_n)$
of objects of $\CC$ and morphisms are tuples $(f_1, \ldots, f_n)$ of morphisms of $\CC$.

\begin{THM}
  Let $\CC_1$ and $\CC_2$ be categories such that $|\hom_{\CC_i}(\calA, \calB)|$ is finite
  for each $i \in \{1, 2\}$ and all $\calA, \calB \in \Ob(\CC_i)$.

  If $\CC_1$ and $\CC_2$ both have the Ramsey property for objects (morphisms) then
  $\CC_1 \times \CC_2$ has the Ramsey property for objects (morphisms).
\end{THM}
\begin{proof}
  The proof is nothing but a reformulation of the proof of Theorem~5 in~\cite[Ch.~5.1]{GRS}.
  Nevertheless, as a demonstration we outline the proof in case of the Ramsey property for objects.

  Take any $k \ge 2$ and $\tilde\calA = (\calA_1, \calA_2)$, $\tilde\calB = (\calB_1, \calB_2)$ in $\Ob(\CC_1 \times \CC_2)$
  such that $\tilde\calA \to \tilde\calB$ and let us show that there is a $\tilde\calC \in \Ob(\CC_1 \times \CC_2)$
  such that $\tilde\calC \longrightarrow (\tilde\calB)^{\tilde\calA}_k$. Take $\calC_1 \in \Ob(\CC_1)$ and
  $\calC_2 \in \Ob(\CC_2)$ so that $\calC_1 \longrightarrow (\calB_1)^{\calA_1}_{k}$ and
  $\calC_2 \longrightarrow (\calB_2)^{\calA_2}_{k^t}$,
  where $t$ is the cardinality of $\binom{\calC_1}{\calA_1}$. Put $\tilde \calC = (\calC_1, \calC_2)$.

  To show that $\tilde\calC \longrightarrow (\tilde\calB)^{\tilde\calA}_k$ take any coloring
  $$
    \chi : \binom{\tilde\calC}{\tilde\calA} \to \{1, \ldots, k\}.
  $$
  Since $\binom{\tilde\calC}{\tilde\calA} = \binom{\calC_1}{\calA_1} \times \binom{\calC_2}{\calA_2}$, the
  coloring $\chi$ uniquely induces the $k^t$-coloring
  $$
    \chi' : \binom{\calC_2}{\calA_2} \to \{1, \ldots, k\}^{\binom{\calC_1}{\calA_1}}
  $$
  of $\binom{\calC_2}{\calA_2}$.
  By construction, $\calC_{2} \longrightarrow (\calB_{2})^{\calA_{2}}_{k^t}$,
  so there is a $w_2 : \calB_2 \to \calC_2$ such that $w_2 \cdot \binom{\calB_2}{\calA_2}$ is $\chi'$-monochromatic.
  Let
  $$
    \chi'' : \binom{\calC_1}{\calA_1} \to \{1, \ldots, k\}
  $$
  be the $k$-coloring of $\binom{\calC_1}{\calA_1}$ defined by
  $$
    \chi''(e_1) = \chi(e_1, e)
  $$
  for some $e \in w_2 \cdot \binom{\calB_2}{\calA_2}$. (Note that $\chi''$ is well defined because
  $w_2 \cdot \binom{\calB_2}{\calA_2}$ is $\chi'$-monochromatic.)
  Since $\CC_1$ has the Ramsey property for objects there is a morphism
  $
    w_1 : \calB_1 \to \calC_1
  $
  such that $w_1 \cdot \binom{\calB_1}{\calA_1}$
  is $\chi''$-monochromatic. It is now easy to show that for $\tilde w = (w_1, w_2)$ we have that
  $\tilde w \cdot \binom{(\calB_1, \calB_2)}{(\calA_1, \calA_2)}$ is $\chi$-monochromatic.
\end{proof}

\begin{COR}\label{cerp.cor.meta}
  $(a)$
  Let $\CC$ be a category whose objects are finite structures and morphisms are embeddings. If
  $\CC$ has the Ramsey property for objects (morphisms) then
  $\CC^n$ has the Ramsey property for objects (morphisms)
  for all $n \in \NN$.

  $(b)$
  Let $\CC$ be a category whose objects are finite structures and morphisms are surjective. If
  $\CC$ has the dual Ramsey property for objects (morphisms) then
  $\CC^n$ has the dual Ramsey property for objects (morphisms)
  for all $n \in \NN$.
\end{COR}

This corrolary can be restated as a metatheorem:
\begin{quote}\it
  Every finite (dual) Ramsey theorem has the finite product version.
\end{quote}

\begin{EX}
  The Finite Dual Ramsey Theorem (Theorem~\ref{cerp.thm.FDRT})
  says that the category $\FinSetSurj$ of finite sets and surjective maps has the dual Ramsey property for objects.
  Therefore, Corollary~\ref{cerp.cor.meta}~$(b)$ implies that for every $n$,
  the category $\FinSetSurj^n$ has the dual Ramsey property for objects. In other words, we have the following
  \emph{finite product dual Ramsey theorem}:
  \begin{quote}
    For all positive integers $s$, $k$, $a_1, \ldots, a_s$, $m_1, \ldots, m_s$
    there exist positive integers $n_1, \ldots, n_s$ such that for all sets
    $C_1, \ldots, C_s$ of cardinalities $n_1, \ldots, n_s$, respectively,
    and every $k$-coloring of the set $\quotient {C_1}{a_1} \times \ldots \times \quotient {C_s}{a_s}$,
    where $\quotient Ca$ is the set of all partitions of $C$ with exactly $a$ blocks,
    there exist a partition $\beta_1$ of $C_1$ with $m_1$ blocks, \ldots,
    a partition $\beta_s$ of $C_s$ with $m_s$ blocks
    such that the following set is monochromatic:
    $
      \big\{
        (\gamma_1, \ldots, \gamma_s) \in \quotient {C_1}{a_1} \times \ldots \times \quotient {C_s}{a_s} :
        \gamma_i \text{ is coarser than } \beta_i \text{ for all } i \in \{1, \ldots, s\}
      \big\}
    $.
  \end{quote}
\end{EX}

\begin{EX}
  Since $\FinSetInj$ and $\FinSetSurj^\op$ both have the Ramsey property for objects (Examples~\ref{cerp.ex.FRP}
  and~\ref{cerp.ex.FDRT}), so does their product. Therefore,
  \begin{quote}
    For all $a, b \in \NN$ and $k \ge 2$
    there exists a $c \in \NN$ such that
    for every set $\calC$ with $|C| = c$ and for every coloring
    $
      \chi : \binom C a \times \quotient{C}{a} \to k
    $
    there is a set $B \subseteq C$ with $|B| = b$ and a partition $\beta$ of $C$ with $b$ blocks
    such that the following set is monochromatic:
    $$\textstyle
      \binom B a \times \{\gamma \in \quotient{C}{a} : \gamma \text{ \textsl{is coarser than} } \beta\}.
    $$
  \end{quote}
\end{EX}

\section{Appendix 2}
\label{cerp.sec.app2}

In the original development in \cite{KPT} the language of the \Fraisse\ class was assumed to be relational.  In fact it is well-known that the results extend to (countable) functional languages by encoding function symbols as relation symbols.

For completeness, we discuss how this is done.  Given a theory $T$ in language $L = \{R_i\}_{i \in I} \cup \{f_j\}_{j \in J}$ form language $L^* = \{R_i\}_{i \in I} \cup \{R_{f_j}\}_{j \in J}$ and the expanstion $T^*$ of $T$ that contains the additional sentences

$$(\ast_j) \ \ \forall \overline{x} (R_{f_j}(x_0,\ldots,x_{n(j)-1},x_{n(j)}) \leftrightarrow f_j(x_0,\ldots,x_{n(j)-1}) = x_{n(j)})$$

Then the discussion on p. 131 of \cite{KPT} needs to be modified only slightly.  The space $X_{L^*}$  of all $L^*$-structures with universe $\NN$ can be identified as

$$X_{L^*} = \prod_i 2^{(\mathbb{N}^{n(i)})} \times \prod_j 2^{(\mathbb{N}^{n(j)+1})},$$

\noindent which is compact, as it is homeomorphic to $2^{\NN}$.  The space $X_L$ of structures may be identified with a compact subset, which is all $L^*$-structures satisfying the sentences $(\ast_j)$.  Universal quantification amounts to taking (countably many) intersections and the quantifier-free conditions define closed subsets under the (usual)  product topology.


\begin{thebibliography}{99}
\bibitem{Bodirsky}
  M.\ Bodirsky.
  New Ramsey Classes from Old.
  The Electronic Journal of Combinatorics, 21(2) (2014) \#P2.22.

\bibitem{burris-sankappanavar}
  S.\ Burris, H.\ P.\ Sankappanavar.
  A Course in Universal Algebra.
  Springer-Verlag, 1981.

%\bibitem{davey-priestley}
%  B.\ A.\ Davey, H.\ A.\ Priestley.
%  Introduction to Lattices and Order (2nd Ed).
%  Cambridge University Press, 2002.

%%Lynn: 13 October
\bibitem{fou}
 W.~L. Fouch{\'e}.
 Symmetries and {R}amsey properties of trees.
 Discrete Mathematics, 197/198 (1999) 325--330.
 %16th British Combinatorial Conference (London, 1997).
%%

\bibitem{GLR}
  R.\ L.\ Graham, K.\ Leeb, B.\ L.\ Rothschild.
  Ramsey's theorem for a class of categories.
  Advances in Math.\ 8 (1972) 417--443; errata 10 (1973) 326--327

\bibitem{GR}
  R.\ L.\ Graham, B.\ L.\ Rothschild.
  Ramsey's theorem for n-parameter sets.
  Tran.\ Amer.\ Math.\ Soc.\ 159 (1971), 257--292

\bibitem{GRS}
  R.\ L.\ Graham, B.\ L.\ Rothschild, J.\ H.\ Spencer.
  Ramsey Theory (2nd Ed).
  John Wiley \& Sons, 1990.

%%Lynn: 13 October
\bibitem{hodges}
    W.\ Hodges
    Model theory.
   Cambridge University Press, 1993.
%%

\bibitem{hu1}
  T.\ K.\ Hu.
  Stone Duality for Primal Algebra Theory.
  Math.\ Z., 110 (1969) 180--198.

\bibitem{hu2}
  T.\ K.\ Hu.
  On the Topological Duality for Primal Algebra Theory.
  Algebra Universalis, 1 (1971) 152--154.

\bibitem{jezek}
  J.\ Je\v zek.
  Universal Algebra (First Edition).
  manuscript, 2008.

\bibitem{KPT}
  A.\ S.\ Kechris, V.\ G.\ Pestov, S.\ Todor\v cevi\'c.
  \Fraisse\ limits, Ramsey theory and topological dynamics of automorphism groups.
  GAFA Geometric and Functional Analysis, 15 (2005) 106--189.

\bibitem{kechris-sokic}
  A.\ Kechris, M.\ Soki\'c.
  Dynamical properties of the automorphism groups of the random poset and random distributive lattice.
  Fund.\ Math.\ 218 (2012), 69--94.

\bibitem{masul-mudri}
  D.\ Ma\v sulovi\'c, N.\ Mudrinski.
  On the Dual Ramsey Property for Finite Distributive Lattices.
  arXiv:1506.01221

\bibitem{mu-pon}
  M.\ M\"uller, A.\ Pongr\'acz.
  Topological Dynamics of unordered Ramsey structures.
  (preprint, 2014)

\bibitem{Nesetril}
  J.\ Ne\v set\v ril.
  Ramsey classes and homogeneous structures.
  Combinatorics, probability and computing, 14 (2005) 171--189.

\bibitem{NR-2}
  J.\ Ne\v set\v ril, V.\ R\"odl.
  Partition (Ramsey) theory and its applications.
  In: Surveys in Combinatorics, Cambridge University Press, Cambridge (1979), 96--156.

\bibitem{vanthe-ufcs}
  L.\ Nguyen Van Th\'e.
  Universal flows of closed subgroups of $S_\infty$ and relative extreme amenability.
  Asymptotic Geometric Analysis, Fields Institute Communications Vol.\ 68, 2013, 229--245


%\bibitem{vanthe-more}
%  L.\ Nguyen Van Th\'e.
%  More on the Kechris-Pestov-Todorcevic correspondence: precompact expansions.
%  Fund.\ Math.\ 222 (2013), 19--47
%
%\bibitem{Nesetril-metric-2}
%  J.\ Ne\v set\v ril. Ramsey classes of topological and metric spaces. Annals of Pure and Applied Logic 143 (2006), 147--154
%
%\bibitem{Nesetril-metric}
%  J.\ Ne\v set\v ril. Metric spaces are Ramsey. European Journal of Combinatorics 28 (2007), 457--468
%
%\bibitem{Nesetril-Rodl}
%  J.\ Ne\v set\v ril, V.\ R\"odl. Partitions of finite relational and set systems.
%  J.\ Comb.\ Th.\ A, 22 (1977), 289--312
%
%\bibitem{Nesetril-Rodl-3}
%  J.\ Ne\v set\v ril, V.\ R\"odl. On a Probabilistic Graph-Theoretical Method.
%  Proc.\ Amer.\ Math.\ Soc.\ 2 (1978), 417--421
%
%\bibitem{Nesetril-Rodl-2}
%  J.\ Ne\v set\v ril, V.\ R\"odl. Combinatorial partitions of finite posets and lattices -- Ramsey lattices.
%  Alg.\ Univers.\ 19 (1984), 106--119
%
%\bibitem{PTW}
%  M.\ Paoli, W.\ Trotter, J.\ Walker. Graphs and orders in Ramsey theory and in dimension theory.
%  In: I.\ Rival, ed, Graphs and Order, Reidel Dordrecht 1984

\bibitem{Pestov-1998}
  V.\ G.\ Pestov.
  On free actions, minimal flows and a problem by Ellis.
  Transactions of the American Mathematical Society, 350 (1998) 4149--4165.

\bibitem{Promel-Voigt-2}
  H.\ J.\ Pr\"omel, B.\ Voigt.
  Recent results in partition (Ramsey) theory for finite lattices.
  Discr.\ Math., 35 (1981), 185--198.

\bibitem{P-V}
  H.\ J.\ Pr\"omel, B.\ Voigt.
  Hereditary attributes of surjections and parameter sets.
  European J.\ Combin.\ 7 (1986), 161--170

\bibitem{Ramsey}
  F.\ P.\ Ramsey.
  On a problem of formal logic.
  Proc.\ London Math.\ Soc.\ 30 (1930), 264--286

%%Lynn: 13 October
\bibitem{scow}
 L. \ Scow.
 Indiscernibles, EM-types, and Ramsey Classes of Trees.
 Notre Dame Journal of Formal Logic, 56 (2015), no. 3, 429-447.

\bibitem{sh90}
 S. \ Shelah.
 Classification Theory (2nd Ed).
 North Holland, 1990.
%%

\bibitem{sokic}
  M.\ Soki\'c.
  Ramsey Properties of Finite Posets.
  Order, 29 (2012) 1--30.

\bibitem{sokic2}
  M.\ Soki\'c.
  Ramsey Properties of Finite Posets II.
  Order, 29 (2012) 31--47.

\bibitem{sokic-phd}
  M.\ Soki\'c. Ramsey property of posets and related structures. PhD thesis, University of Toronto, 2010

\bibitem{solecki-dual}
  S.\ Solecki.
  Dual Ramsey theorem for trees.
  arXiv:1502.04442v1

\bibitem{zucker1}
  A.\ Zucker.
  Topological dynamics of closed subgroups of $S_\infty$.
  arXiv:1404.5057

\end{thebibliography}
\end{document}